\documentclass[11 pt, final]{article}
\title{The large scale geometry of strongly aperiodic subshifts of finite type}
\author{David Bruce Cohen}

\usepackage{fullpage}
\usepackage{graphicx}
\usepackage{amsmath}
\usepackage{amssymb}
\usepackage{amsthm}
\usepackage{epsfig,pinlabel}
\usepackage{amsfonts}
\usepackage{amscd}

\DeclareMathOperator{\Nn}{\ensuremath{\mathcal{N}}}

\newcommand{\Cay}{Cay}

\newcommand{\RRD}{R\Delta}

\DeclareMathOperator{\ZZ}{\ensuremath{\mathbb{Z}}}

\DeclareMathOperator{\AAA}{\ensuremath{\mathbb{A}}}
\DeclareMathOperator{\NN}{\ensuremath{\mathbb{N}}}

\newcommand\la{\ensuremath{\langle}}
\newcommand\ra{\ensuremath{\rangle}}
\newcommand\To{\ensuremath{\rightarrow}}
\newcommand\smin{\ensuremath{\setminus}}

\newcommand\omg{\ensuremath{\sigma}}

\newcommand\gax{\ensuremath{g_{ax}}}
\newcommand\fS{\ensuremath{\mathcal{S}}}
\newcommand\fSp{\ensuremath{\mathcal{S}^{\prime}}}
\newcommand\frm{\ensuremath{\mathfrak{m}}}
\newcommand\omo{\ensuremath{\sigma_{0}}}
\newcommand\BnG{\ensuremath{B_{G}(n,1_{G})}}

\newcommand\GmS{\ensuremath{\Gamma_{S}}}
\newcommand\fCi{\ensuremath{\mathcal{C}_{i}}}

\newcommand\BHN{\ensuremath{B_{H}(N,1_{H})}}

\newcommand\oml{\ensuremath{\sigma_{\ell}}}
\newcommand\omd{\ensuremath{\sigma_{d}}}
\newcommand\lFf{\ensuremath{\ell_{Ff}}}

\newcommand\Fom{\ensuremath{F_{\sigma}}}
\newcommand\Fomg{\ensuremath{F_{\sigma g}}}

\DeclareMathOperator{\Lip}{Lip}

\DeclareMathOperator{\QIP}{QIP}
\DeclareMathOperator{\Stab}{Stab}

\theoremstyle{plain}
\newtheorem{theorem}{Theorem}[section]
\newtheorem{lemma}[theorem]{Lemma}

\newtheorem{conjecture}[theorem]{Conjecture}

\newtheorem{proposition}[theorem]{Proposition}
\newtheorem{corollary}[theorem]{Corollary}
\newtheorem{definition}[theorem]{Definition}

\begin{document}

\maketitle
\begin{abstract}
A subshift on a group $G$ is a closed, $G$-invariant subset of $A^{G}$, for some finite set $A$. It is said to be a subshift of finite type (SFT) if it is defined by a finite collection of ``forbidden patterns", to be strongly aperiodic if all point stabilizers are trivial, and weakly aperiodic if all point stabilizers are infinite index in $G$. We show that groups with at least $2$ ends have a strongly aperiodic SFT, and that having such an SFT is a QI invariant for finitely presented torsion free groups. We show that a finitely presented torsion free group with no weakly aperiodic SFT must be QI-rigid. The domino problem on $G$ asks whether the SFT specified by a given set of forbidden patterns is empty. We show that decidability of the domino problem is a QI invariant.
\end{abstract}

\section{Introduction}
\label{section:Introduction}
Recall that a topological dynamical system is a pair $(\Omega, G)$ where $G$ is a group acting by homeomorphisms on the compact space $\Omega$. For instance, if $A$ is a finite discrete set, then the group $G$ acts (on the right) on the compact space $A^{G}$ by homeomorphisms via
$$(\sigma\cdot h)(g)=\sigma(hg).$$
This action makes the pair $(A^{G},G)$ into a topological dynamical system called the right shift. When $G=\ZZ$, an element $h$ of $G$ acts on a bi-infinite word $\sigma\in A^{G}$ by ``shifting" it, whence the name. A closed, $G$-invariant subset of $A^{G}$ is known as a subshift. To say that a subshift $X$ codes a dynamical system $(\Omega,G)$ means that there exists a continuous $G$-equivariant surjection from $X$ to $\Omega$.

\paragraph{Subshifts of finite type (see \cite[\S 2]{cp}).} How would one construct a subshift? The simplest idea is to start with a closed set $C$ of $A^{G}$ and intersect its $G$-translates. The most important case of this construction arises when $C$ is determined by finitely many coordinates.
\begin{definition}
\label{definition:ssft}
Let $A$ be a finite set and $G$ a group. If $S$ is a finite subset of $G$ and $L$ a subset of $A^{S}$, then the clopen set
$$\{\sigma\in A^{G}:\sigma|_{S}\in L\}$$
is known as a cylinder set. If $C$ is a cylinder set, then the set $X$ given by $\bigcap_{g\in G}(C\cdot g)$ is called a subshift of finite type. We say that $X$ is defined on $S$. If $F$ is a finite set, then $\alpha\in A^{F}$ is called a forbidden pattern for $X$ if it is never equal to $(\sigma\cdot g)|_{F}$ for any $\sigma\in X$.
\end{definition}

We say that a subshift of finite type $X\subset A^{G}$ is defined by a finite collection $\mathcal{F}$ of forbidden patterns $\alpha_{i}:F_{i}\To A$ if $X$ is exactly the set of $\sigma\in A^{G}$ such that $(\sigma\cdot g)|_{F_{i}}$ is not equal to $\alpha_{i}$ for any $i$ and any $g\in G$.

\subsection{Aperiodicity and the domino problem.}
\label{subsection:wang}
Given a finite set of forbidden patterns $\mathcal{F}$, it is entirely possible that the subshift of finite $X_{\mathcal{F}}$ defined by $\mathcal{F}$ is empty.

\begin{definition}
\label{definition:domino}
Let $G$ be a finitely generated group. We say that $G$ has decidable domino problem if there exists an algorithm which takes as input a finite set of forbidden patterns $\mathcal{F}$ and determines whether the subshift they define is empty.
\end{definition}

\paragraph{The domino problem for $\mathbb{Z}$.} Suppose we are given a finite set of forbidden patterns $\mathcal{F}$ defining a subshift of finite type $X_{\mathcal{F}}$ over $\mathbb{Z}$. By compactness $X_{\mathcal{F}}$ is empty only if there is some $n$ such that every $A$-coloring of the $n$-ball in $\mathbb{Z}$ includes some forbidden pattern. Hence, if $X_{\mathcal{F}}$ is empty, a Turing machine may discover this fact in finite time.

On the other hand, if $X_{\mathcal{F}}$ is nonempty, one would like to certify nonemptiness by finding a ``constructible" $\sigma\in A^{G}$ which can be proven to lie in $X_{\mathcal{F}}$. The simplest constructible elements of $A^{\mathbb{Z}}$ are periodic---i.e., fixed by translation by some $n$ and hence of the form $\ldots www\ldots$ for some word $w$ of length $n$. In fact, every nonempty subshift of finite type over $\mathbb{Z}$ contains such elements (Figure \ref{figure:z} in \S\ref{section:ends}), and the domino problem for $\ZZ$ can be decided by finding periodic elements. One can also discuss periodicity over more general groups.

\begin{definition}
\label{definition:sa}
Let $G$ be a finitely generated group, $X\subset A^{G}$ a nonempty subshift, and $\sigma$ a point of $X$. Then $\sigma$ is said to be periodic if it has nontrivial stabilizer in $G$, and is said to be $g$-periodic for any $g\in \Stab_{G}\sigma$.  If $X$ contains no periodic points, then $X$ is said to be strongly aperiodic. If $X$ contains no points with finite index stabilizer, $X$ is said to be weakly aperiodic.
\end{definition}

\begin{figure}[t]
\labellist
\small\hair 2pt

\endlabellist

\centering
\centerline{\psfig{file=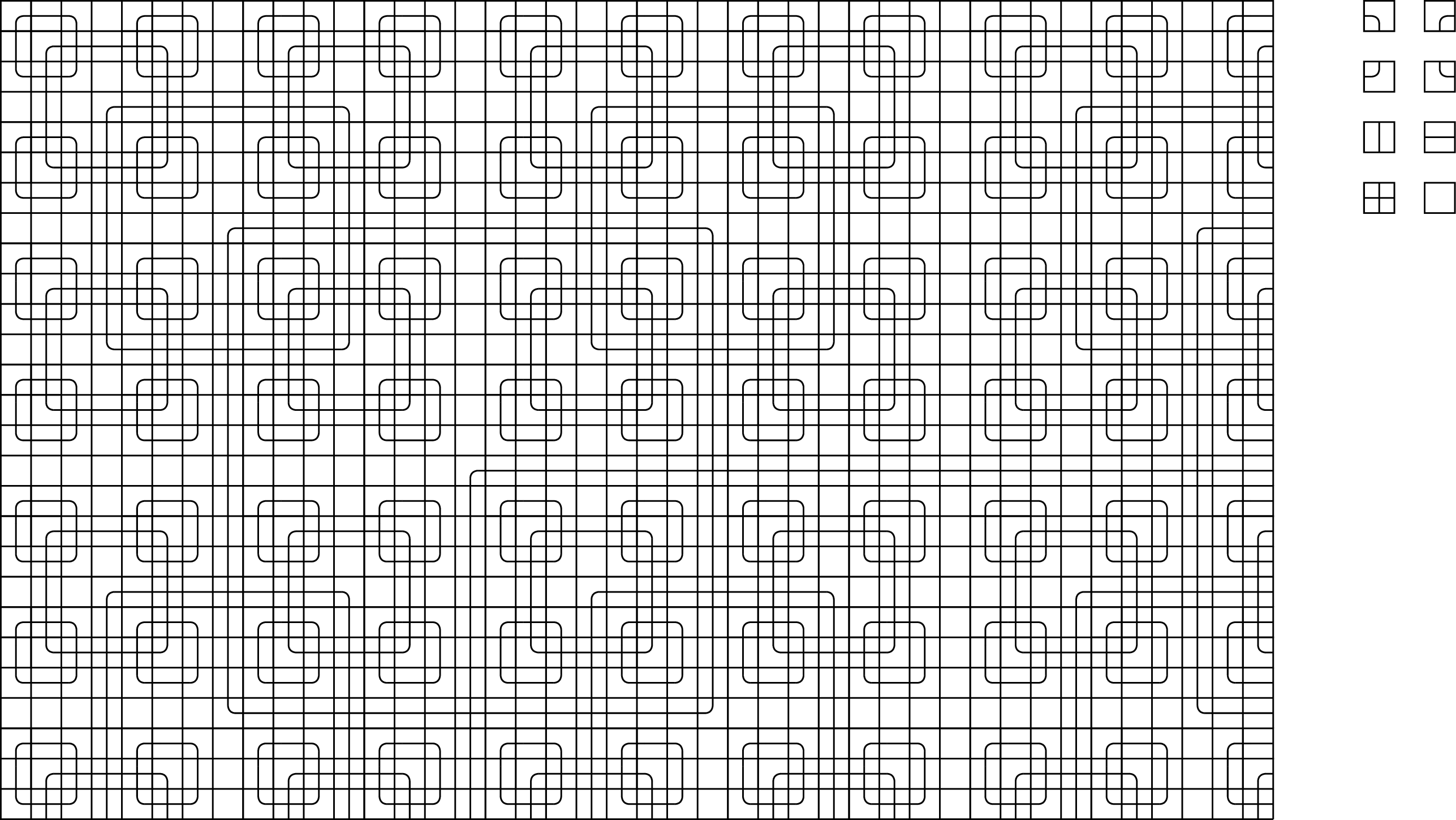,scale=60}}
\caption{An interesting tiling of the plane using the set $A$ consisting of the eight tiles depicted at right.  The orbit closure in $A^{\ZZ^{2}}$ of the given pattern is strongly aperiodic, and is coded by a subshift of finite type called the Robinson tiling, in which the tiles carry slightly more data.}
\label{figure:robinson}
\end{figure}

\paragraph{The domino problem for $\mathbb{Z}^{2}$ (Wang tiles).} Suppose we are given a finite set $A$ of $1\times 1$ square tiles, such that each edge of each tile is assigned some color. We say that $A$ tiles the plane if we may fill out the entire plane with copies of these tiles such that neighboring edges have the same color. In the simplest examples, a collection of tiles which successfully tiles the plane can do so periodically. Wang conjectured that this is always the case---i.e., that $A$ tiles periodically if it tiles at all. He observed that if his conjecture were true, then a Turing machine could, given $A$, decided whether $A$ tiles the plane.

Wang's conjecture was disproved by his student Berger \cite{berger}, who showed that no Turing machine can decide whether a given set of tiles can tile the plane, and found an explicit set $A$ of tiles which can tile the plane, but cannot do so periodically. Since then, many people have obtained interesting tile sets with this property.  Our favorite is the Robinson tiling \cite{robinson}, which codes the orbit closure of the pattern depicted in Figure \ref{figure:robinson}. Of course, if $A$ tiles the plane, but cannot tile it periodically, then we obtain a strongly aperiodic subshift of finite type inside $A^{\ZZ^{2}}$, where the forbidden patterns consist of pairs of adjacent tiles with non-matching edges.

\subsection{Known results}
\label{subsection:knownresults}
Consider the following questions.
\begin{itemize}
\item Which groups have decidable domino problem?
\item Which groups admit strongly aperiodic subshifts of finite type?
\item Which groups admit weakly aperiodic subshifts of finite type?
\end{itemize}
All three of these questions are open, and our main theorems will concern their answers. Before explaining our results, here is a brief survey of known work and conjectured answers.

\paragraph{Which groups have decidable domino problem?} Berger's result shows that $\ZZ^{2}$ has undecidable domino problem, whereas we have remarked that $\ZZ$ is known to have decidable domino problem. Aubrun and Kari have shown that the Baumslag Solitar groups have undecidable domino problem \cite{aubrunkari}. Ballier and Stein \cite{ballierstein}, building on results from several authors \cite{mullerschupp1}\cite{mullerschupp2}\cite{kuskelohrey}\cite{jeandeltheyssier}, observe that every virtually free group has decidable domino problem, and conjecture that these are the only such groups. 

\paragraph{Which groups have strongly aperiodic subshifts of finite type?} Berger showed that $\ZZ^{2}$ has a strongly aperiodic subshift of finite type \cite{berger}. Many other groups are known to admit such subshifts, including higher rank free abelian groups \cite{kariculik}, solvable Baumslag Solitar groups \cite{aubrunkari}, the integral Heisenberg group \cite{ssu}, cocompact lattices in higher rank simple Lie groups \cite{mozes}, and the direct product of Thompson's group $T$ with $\ZZ$ \cite{jeandel}. Forthcoming work of the author and Goodman-Strauss will show that surface groups also have such subshifts \cite{chaim}.

On the other hand, no free group has a strongly aperiodic subshift of finite type \cite{piantadosi}. When we began writing this paper, this was the only known negative result. However, Jeandel has since discovered a remarkable obstruction to admitting a strongly aperiodic subshift of finite type, as discussed in Subsection \ref{subsection:ggt}.

\paragraph{Which groups have weakly aperiodic subshifts of finite type?} We have remarked that $\ZZ$ has no weakly aperiodic subshift of finite type, but many groups are known not to share this property. In particular, nonamenable groups \cite{bw}, free abelian groups \cite{berger}, Baumslag Solitar groups \cite{aubrunkari}, and Grigorchuk's group \cite{marcinkowskinowak} cannot have weakly aperiodic subshifts of finite type. Carroll and Penland have conjectured that a group has a weakly aperiodic subshift of finite type if and only if it is virtually cyclic \cite{carpen}.

\subsection{Endedness and QI invariance}
\label{subsection:ggt}
We will prove four theorems which show that for a given group $G$, the answers to the above questions (decidability of the domino problem and existence of strongly/weakly aperiodic subshifts of finite type) are closely related to the geometry of $G$---meaning the geometry of its Cayley graph. Recall that the Cayley graph of a group $G$ with respect to a generating set $S$ is the graph whose vertex set is $G$, with an edge between vertices $g$ and $h$ whenever $gs=h$ for some $s$ in $S$ (Figure \ref{figure:cayley} depicts some Cayley graphs). The argument used by Piantadosi \cite{piantadosi} to show that free groups have no strongly aperiodic subshift of finite type is based on the fact that the Cayley graph of a free group may be disconnected by removing a sufficiently large ball around a point.  The following definition captures this idea.

\begin{definition}
\label{definition:endedness}
Let $S$ be a finite generating set for a group $G$. The number of ends of $G$ is defined to be the limit as $n$ goes to infinity of the number of unbounded connected components of $G\smin B_{n}$, where $B_{n}$ is the ball of radius $n$ around $1_{G}$ in the Cayley graph of $G$. It is understood that this limit is often infinite.
\end{definition}

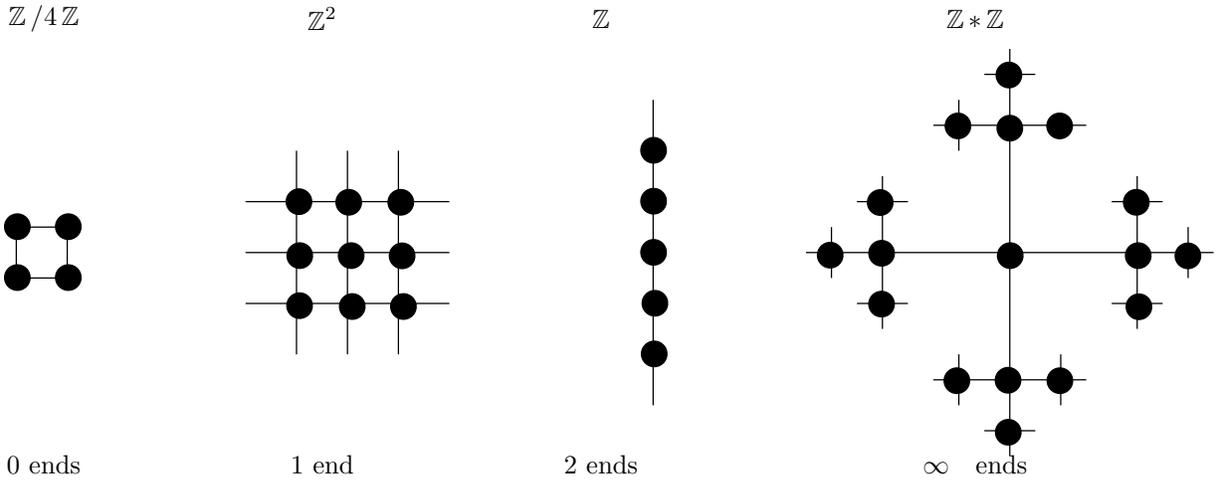
\begin{figure}[t]
\labellist
\small\hair 2pt

\pinlabel $\ZZ/4\ZZ$ at 25 275
\pinlabel $\ZZ^{2}$ at 200 275
\pinlabel $\ZZ$ at 375 275
\pinlabel $\ZZ\ast\ZZ$ at 610 275
\pinlabel $\text{0 ends}$ at 25 -5
\pinlabel $\text{1 end}$ at 200 -5
\pinlabel $\text{2 ends}$ at 375 -5
\pinlabel $\infty\quad\text{ends}$ at 610 -5
\endlabellist

\centering
\centerline{\psfig{file=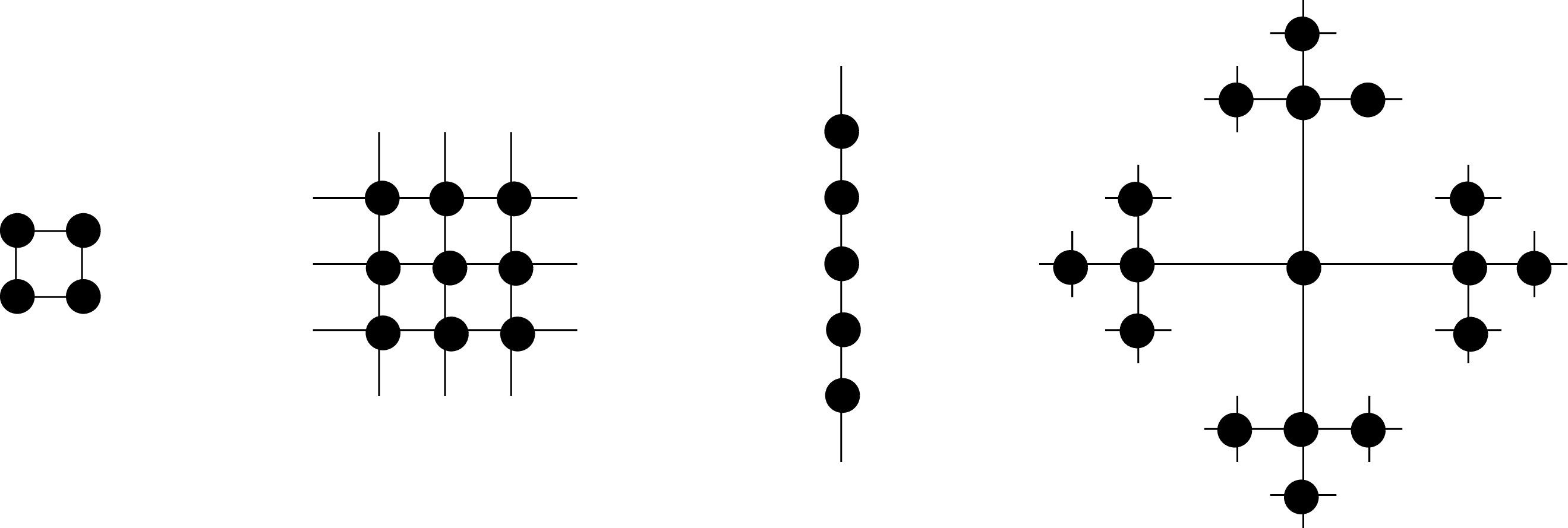,scale=60}}
\caption{Some Cayley graphs of groups with respect to their standard generating sets.}
\label{figure:cayley}
\end{figure}

The number of ends of $G$ is invariant under changing the generating set $S$. The point is that, if $S^{\prime}$ is some other finite generating set, then the Cayley graphs associated to $S$ and $S^{\prime}$ are quasi isometric (Definition \ref{definition:qi}) and the the number of ends is a ``QI invariant".  Hopf discovered that the number of ends of a group is either $0,1,2,$ or $\infty$ \cite{hopf}, and Stallings \cite{stallings} showed that a group has at least $2$ ends if and only if it splits nontrivially as an amalgamated free product or HNN extension over a finite group (of course, $G$ has $0$ ends if and only if it is finite). In the case of torsion free groups, having one end is equivalent to being neither cyclic nor a free product.

Consider the groups $G$ listed above for which we know whether or not $G$ admits strongly aperiodic subshifts of finite type.  Of the groups not admitting such subshifts, $\ZZ$ has $2$ ends, and higher rank free groups have infinitely many ends. On the other hand, the groups known to have such subshifts---the Heisenberg group, Thompson's $T$ direct product $\ZZ$, cocompact lattices, and free abelian groups---are all one ended. We shall prove the following theorem.

\begin{theorem}
\label{theorem:mainthmends}
If $G$ is a finitely generated group with at least $2$ ends, then $G$ does not admit a strongly aperiodic subshift of finite type.
\end{theorem}

In an earlier version of this paper, we conjectured the converse.

\begin{conjecture}
\label{conjecture:main}
Let $G$ be an infinite, finitely generated group. Then $G$ admits a strongly aperiodic subshift of finite type if and only if it is one ended.
\end{conjecture}

This conjecture is now known to be false by work of Jeandel \cite[Corollary 2.7]{jeandel}, which shows that groups with undecidable word problem cannot have strongly aperiodic subshifts of finite type. There are many known examples of one ended groups with undecidable word problem.

\paragraph{QI invariance.} In order to state our other theorems, we require the following definition, which was alluded to above.
\begin{definition}
\label{definition:qi}
A map $f:X\To Y$ between metric spaces $X$ and $Y$ is said to be an $n$-quasi isometric embedding if for any points $x_{1},x_{2}\in X$,
$$\frac{d(x_{1},x_{2})}{n}-n\leq d(f(x),f(y))\leq nd(x_{1},x_{2})+n.$$
It is said to be $n$-quasi surjective if the $n$ neighborhood of $f(X)$ equals all of $Y$.  We say that $f$ is a quasi isometry if (for some $n$) it is an $n$-quasi surjective $n$-quasi isometric embedding.
\end{definition}

Two spaces are said to be quasi isometric if there exists a quasi isometry between them, and it is easily seen that this is an equivalence relation. This equivalence relation is interesting for Cayley graphs, which may be metrized by taking each edge to have length $1$.  As mentioned above, if $S,S^{\prime}$ are finite generating sets for a group $G$, then the Cayley graph of $G$ with respect to $S$ is quasi isometric to the Cayley graph of $G$ with respect to $S^{\prime}$.  The following definition is the basis of the subject of geometric group theory.

\begin{definition}
\label{definition:qi2}
We say that finitely generated groups $G$ and $H$ are quasi isometric if their Cayley graphs are quasi isometric. If $\eta$ is an invariant of groups such that $\eta(G)=\eta(H)$ whenever $G$ and $H$ are quasi isometric, we say that $\eta$ is a QI invariant.
\end{definition}

As remarked above, the number of ends of a group is the prototypical QI invariant.  For another example, finite presentation is a QI invariant---if $G$ is finitely presented and $G$ is quasi isometric to $H$, then $H$ is finitely presented. We will prove that for finitely presented groups, having decidable domino problem is a QI invariant, as is having a strongly aperiodic subshift of finite type (at least under the hypothesis of torsion freeness).

\begin{theorem}
\label{theorem:mainthmdomino}
Let $G$ and $H$ be finitely presented groups with $G$ quasi isometric to $H$.  Then $G$ has decidable domino problem if and only if $H$ does.
\end{theorem}

\begin{theorem}
\label{theorem:mainthmstrong}
Let $G$ and $H$ be torsion free finitely presented groups, and suppose that $G$ is quasi isometric to $H$.  Then $G$ admits a strongly aperiodic subshift of finite type if and only if $H$ does.
\end{theorem}

\paragraph{Remark.} We note that Carroll and Penland have shown independently that having a strongly aperiodic subshift of finite type is a commensurability invariant \cite{carpen}, even without assuming torsion freeness or finite presentation.  Two groups $G$ and $H$ are said to be commensurable if some finite index subgroup of $G$ is isomorphic to some finite index subgroup of $H$. If $G$ and $H$ are commensurable, then they are quasi isometric to each other, but there are many examples of pairs of groups which are quasi isometric but not commensurable.\\

With regards to weak aperiodicity, we prove an even stronger result. A set of groups $\mathcal{S}$ is said to be QI-rigid if whenever a group $G$ is quasi isometric to some element of $\mathcal{S}$, then $G$ is actually commensurable with some element of $\mathcal{S}$. We will show that if $G$ is finitely presented with no weakly aperiodic subshift of finite type, then the singleton set $\{G\}$ is QI-rigid, so long as $G$ is torsion free.

\begin{theorem}
\label{theorem:mainthmweak}
If $G$ is a finitely presented group with no weakly aperiodic subshift of finite type, and $H$ is quasi isometric to $G$, then there exist finite index subgroups $G_0\subset G$ and $H_0\subset H$ such that $H_0$ is isomorphic to the quotient of $G_0$ by a finite group. 
\end{theorem}

\subsection{Organization.}
\label{subsection:outline}
The paper is organized as follows. Section \ref{section:ends} gives the proof of Theorem \ref{theorem:mainthmends}. Section \ref{section:dshift} defines the notion of the derivative of an $n$-Lipschitz function on a finitely generated group, which will be crucial in proving Theorems \ref{theorem:mainthmweak}, \ref{theorem:mainthmstrong}, and \ref{theorem:mainthmdomino}. In this section, we prove Theorem \ref{theorem:derivativeshift}, which states that the collection of derivatives of $n$-Lipschitz functions on a finitely presented group forms a subshift of finite type.

Section \ref{section:qi} will show that there is a subshift of finite type parameterizing certain quasi isometries between two groups $G$ and $H$ (Lemma \ref{lemma:qipairshift}). Theorem \ref{theorem:mainthmweak} will follow as Corollary \ref{corollary:weakaperiodicity}. Section \ref{section:pullbacks} will describe how to construct a subshift on a group $G$ parameterizing pairs $(f,\sigma)$ where $f:G\To H$ is a quasi isometry and $\sigma$ is a configuration of some subshift of finite type $X\subset A^{H}$ (Lemma \ref{lemma:pullbackshift}).  Theorems \ref{theorem:mainthmstrong} and \ref{theorem:mainthmdomino} will follow as Corollaries \ref{corollary:strongaperiodicity} and \ref{corollary:dominoproblem}.

\subsection{Notation.}
Throughout this paper, a finitely generated group will be equipped with the word metric with respect to some fixed generating set (the metric inherited from the Cayley graph). Hence, when we say that a function $f:G\To H$ between two groups is a quasi isometry or Lipschitz map, we mean that it has the given property with respect to these word metrics. If $g$ is an element of a group $G$, then $B(n,g)$ will denote the $n$-ball around $g$ in $G$---i.e., the set of $g'\in G$ such that $d(g,g')\leq n$. Similarly, if $F$ is a subset of $G$, then $\Nn_{m}F$ will denote the $n$-neighborhood of $F$---i.e., the set of all $g\in G$ such that there exists $f\in F$ with $d(g,f)\leq n$. Finally, the identity element of a group $G$ will be denoted as $1_{G}$, or just $1$ when there is no possibility of confusion.

\paragraph{Acknowledgments.} We wish to thank Andy Putman for his guidance, Ay\c{s}e \c{S}ahin for discussing her work with us, and Ilya Kapovich for his thoughtful comments on early drafts of this paper. We also wish to thank Andrew Penland for explaining his results to us, and Danijela Damjanovich for hosting the 2014 Rice Dynamics Meeting. We especially wish to thank Tullio Ceccherini-Silberstein for suggesting that we should be able to prove Theorem \ref{theorem:mainthmdomino} and Yves Cornulier for pointing out some counterexamples to a conjecture which appeared in an earlier version of this paper. We also wish to thank the organizers of the conference ``Growth, Symbolic Dynamics and Combinatorics of Words", supported by ERC starting grant 257110 ``RaWG". This work has been supported by NSF award 1148609.

\section{Groups with at least two ends do not admit a strongly aperiodic subshift of finite type.}
\label{section:ends}
We now prove theorem \ref{theorem:mainthmends}. 
\begin{theorem}
\label{theorem:ends}
Let $G$ be a finitely generated group with at least $2$ ends.  Let $X\subset A^{G}$ be a nonempty subshift of finite type.  Then there exists $\sigma\in X$ and $g\in G$ not equal to $1_{G}$ such that $\sigma\cdot g=\sigma$.
\end{theorem}

\paragraph{An example: $G=\mathbb{Z}$.} We begin by illustrating the proof in a special case (see Figure \ref{figure:z}).  Assume $G=\mathbb{Z}$ and $X\subset A^{G}$ a non empty subshift of finite type. Suppose that $X$ is defined on $B=\BnG$, meaning that to determine whether $\sigma\in A^{G}$ is an element of $X$, we just need to check that the set
$$\{\sigma\cdot g|_{B}:g\in\ZZ\}$$
contains no forbidden pattern. Since we assumed that $X$ is nonempty, there exists $\sigma_{0}\in X$.  We observe that there must exist $m_{1},m_{2}\in \ZZ$ such that $m_{2}-m_{1}>2n$ and $\sigma_{0}\cdot m_{1}|_{B}=\sigma_{0}\cdot m_{2}|_{B}$. We will find $\sigma\in X$ such that $\sigma$ is $m_{2}-m_{1}$-periodic, meaning that $\sigma\cdot (m_{2}-m_{1})=\sigma$.

Let $\fS=\{m_{1}-n,m_{1}-n+1,\ldots,m_{2}-n-1\}$ and $\fSp=\fS\cup (m_{2}+B)$.  Let $\frm:G\rightarrow\fS$ be specified by $\frm(x)\equiv x$ mod $m_{2}-m_{1}$ for every $x\in G$. We define $\sigma$ to be $x\mapsto \sigma_{0}(\frm(x))$. Manifestly, $\sigma$ is $m_{2}-m_{1}$-periodic. To show that $\sigma$ is in $X$, we start with the following observations.
\begin{itemize}
\item[(a)] On $\fSp$, the functions $\sigma$ and $\omo$ agree. To see this, note that if $x\in \fS$, then $\sigma(x)=\omo(x)$ by definition, and if $x\in \fSp\smin \fS$, then $x\in m_{2}+B$, so $\frm(x)=x-(m_{2}-m_{1})$ and
$$\sigma(x)=\omo(x-(m_{2}-m_{1}))=\omo(x)$$
by our assumption that $\omo\cdot m_{1}|_{B}=\omo\cdot m_{2}|_{B}$.
\item[(b)] For all $x\in G$, there exists some $k\in \ZZ$ such that
$$k(m_{2}-m_{1})+x+B\subset \fSp.$$
This follows from the fact that $\fSp$ contains the $n$-neighborhood of $\{m_{1},m_{1}+1,\ldots,m_{2}-1\}$ which is a complete set of coset representatives mod $m_{2}-m_{1}$. 
\end{itemize}

We now show $\sigma\in X$. Given $x\in G$, choose $k$ as in (ii) above. Then by periodicity and (i),
$$\sigma\cdot x|_{B}=\sigma\cdot (k(m_{2}-m_{1})+x+B)|_{B}
=\omo\cdot(k(m_{2}-m_{1})+x)|_{B}.$$
It follows that $\sigma\in X$.

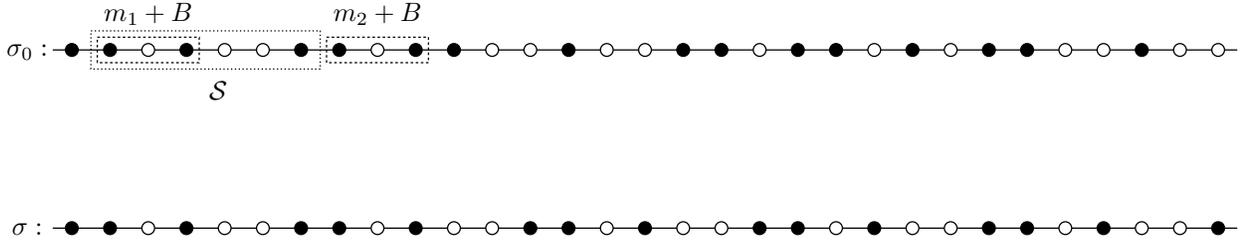
\begin{figure}[t]
\labellist
\small\hair 2pt

\pinlabel $m_{1}+B$ at 60 139
\pinlabel $m_{2}+B$ at 204 139
\pinlabel $\fS$ at 104 91
\pinlabel $\omo :$ at -16 115
\pinlabel $\sigma :$ at -16 4


\endlabellist

\centering
\centerline{\psfig{file=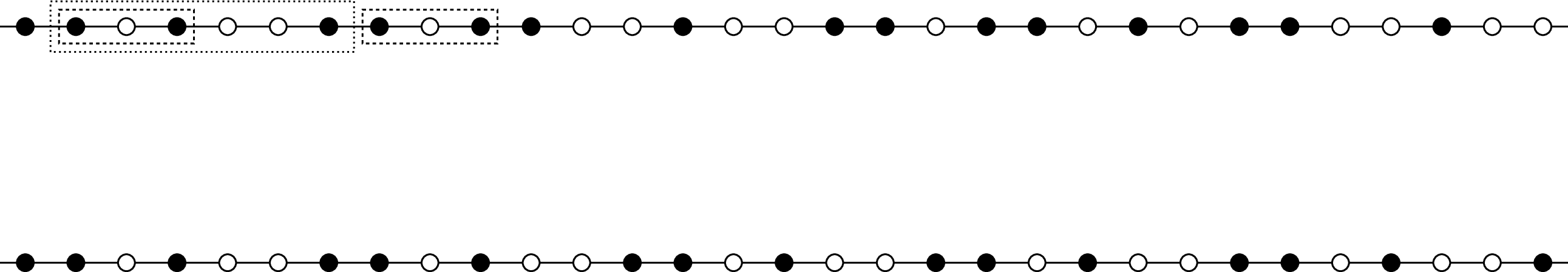,scale=60}}
\caption{Given $\sigma_{0}$, a point of $X$, a subshift of finite type on $\ZZ$, one can construct a periodic point $\sigma$ of $X$ by repeating the pattern found in $\omo$ between two balls $m_{1}+B$ and $m_{2}+B$ on which $\omo$ has the same behavior.}
\label{figure:z}
\end{figure}

\paragraph{The general case.} From here on, we will assume that $G$ is a group with at least $2$ ends, so that for $n$ sufficiently large, $\GmS\smin|\BnG|$ has at least $2$ unbounded connected components. 
The following definition will be crucial.

\begin{definition}
\label{definition:separation}
Let $B_{0},B_{1},B_{2}$ be finite subsets of $G$ such that each $|B_{i}|$ is connected. We say that $B_{1}$ separates $B_{0}$ from $B_{2}$ when $B_{1}$ and $B_{2}$ lie in distinct connected components of $\GmS\setminus |B_{0}|$.
\end{definition}

For example, in $\ZZ$, $b+B(n,1)$ separates $a+B(n,1)$ from $c+B(n,1)$ when $c-b>2n$ and $b-a>2n$.  The following lemma encodes some trivial observations about separation.
\begin{lemma}
\label{lemma:separation}
Suppose that $B_{1}$ separates $B_{0}$ from $B_{2}$.
\begin{itemize}
\item[(a)] If $g\in G$, then $gB_{1}$ separates $gB_{0}$ from $gB_{2}$.
\item[(b)] If $\mathcal{C}$ is an unbounded component of
$$\GmS\smin (|B_{0}|\cup|B_{1}|\cup|B_{2}|),$$
then $\mathcal{C}$ cannot touch both $B_{0}$ and $B_{2}$.
\item[(c)] If we also know that $B_{2}$ separates $B_{1}$ from some finite $B_{3}$, then it follows that $B_{0}$ and $B_{3}$ are separated by $B_{i}$ if $i$ is $1$ or $2$.
\end{itemize}
\end{lemma}

\begin{proof}
Part (a) is trivial.

To see part (b), observe that if $\mathcal{C}$ were an unbounded component which touched both $B_{0}$ and $B_{2}$, then we could find a path in $\mathcal{C}$ joining a vertex of $B_{0}$ to a vertex of $B_{2}$.  Hence, $B_{0}$ and $B_{2}$ would lie in the same connected component of $\GmS\smin|B_{1}|$ (whichever one contains $\mathcal{C}$,) contrary to the definition.

To obtain part (c), we reason as follows. Because $\GmS$ is connected, there exists a path in $\GmS$ from $B_{3}$ to $B_{1}$, but any such path must go through $B_{2}$ because $B_{2}$ separates $B_{3}$ from $B_{1}$. Hence, $B_{2}$ and $B_{3}$ are in the same connected component of $\GmS\smin |B_{1}|$, and therefore $B_{3}$ and $B_{0}$ are in different connected components of $\GmS\smin|B_{1}|$ as desired since $B_{1}$ separates $B_{0}$ from $B_{2}$. The same argument shows that $B_{2}$ separates $B_{0}$ from $B_{3}$.
\end{proof}

We now define the notion of an $n$-axial element $g\in G$. Intuitively (if not in reality,) left multiplication by such an element drags the Cayley graph of $G$ along some axis.


\begin{definition}
\label{definition:naxial}
Let $n$ be a natural number. We say that $g\in G$ is $n$-axial if, for all integers $a<b<c$, we have that $g^{b}\BnG$ separates $g^{a}\BnG$ from $g^{c}\BnG$.
\end{definition}

In $\ZZ$, an element $g$ is $n$-axial exactly when it has absolute value greater than $2n$. We now prove that every group with at least two ends has an $n$-axial element for sufficiently large $n$.

\begin{lemma}
\label{lemma:naxialexistence}
Under our standing assumption that $G$ is a finitely generated group with at least $2$ ends, there exists some $N_{G}$ such that for any $n\geq N_{G}$, there exists an $n$-axial $g\in G$.
\end{lemma}

\begin{proof}
(See the potentially deceptive Figure \ref{figure:axial1}). Suppose $n$ is large enough that $\GmS\smin|\BnG|$ has at least two unbounded components, and write $B$ for $\BnG$. Choose $x,y\in G$ such that each has norm greater than $2n$ and $x$ and $y$ lie in distinct unbounded components of $\GmS\smin|B|$. Manifestly, $B$ separates $xB$ from $yB$, and $B$ also separates $x^{-1}B$ from $y^{-1}B$ since $S$ is assumed symmetric, so a path from $x^{-1}$ to $y^{-1}$ which did not pass through $B$ could be reflected to get a path from $y$ to $x$ not passing through $B$.  We will see that $x^{-1}y$ is $n$-axial.

Inductively define a biinfinite sequence $B_{i}$ of finite subsets of $G$ by setting $B_{0}=B$ and $B_{1}=x^{-1}B$, and mandating that $B_{i+2}=x^{-1}yB_{i}$ for all integers $i$. We know that $B_{0}$ separates $B_{-1}=y^{-1}B$ from $B_{1}$, and also that $B_{1}$ separates $B_{0}$ from $B_{2}=x^{-1}yB$ (by translating $xB,B,yB$ by $x^{-1}$.)  Hence Lemma \ref{lemma:separation}(a) gives us that $B_{i}$ separates $B_{i-1}$ from $B_{i+1}$ for all $i$. But then part (c) of the lemma says that (in particular) $B_{2b}$ separates $B_{2a}$ from $B_{2c}$ whenever $a<b<c$ are integers. I.e., $x^{-1}y$ is n-axial.
\end{proof}

\begin{figure}[t]
\labellist
\small\hair 2pt

\pinlabel $xB$ at 320 152
\pinlabel $yB$ at 449 152

\pinlabel $y^{-1}xy^{-1}xB$ at 128 129
\pinlabel $y^{-1}xB$ at 257 129
\pinlabel $B$ at 385 89
\pinlabel $x^{-1}yB$ at 513 129
\pinlabel $x^{-1}yx^{-1}yB$ at 642 129

\pinlabel $(y^{-1}x)^{2}y^{-1}B$ at 64 -13
\pinlabel $y^{-1}xy^{-1}B$ at 192 -13
\pinlabel $y^{-1}B$ at 320 27
\pinlabel $x^{-1}B$ at 449 27
\pinlabel $x^{-1}yx^{-1}B$ at 577 -13
\pinlabel $(x^{-1}y)^{2}x^{-1}B$ at 705 -13

\pinlabel $\text{?}$ at 350 89


\endlabellist

\centering
\centerline{\psfig{file=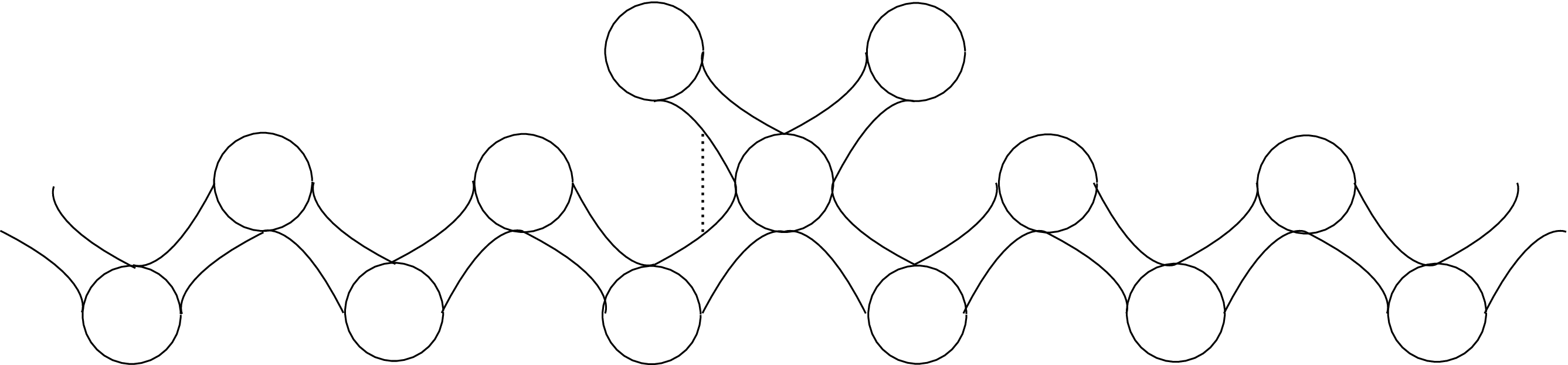,scale=60}}
\caption{Constructing an $n$-axial element.  The question mark indicates one possible way the diagram can be misleading: it is possible that $xB$ and $y^{-1}B$ are in the same connected component.}
\label{figure:axial1}
\end{figure}
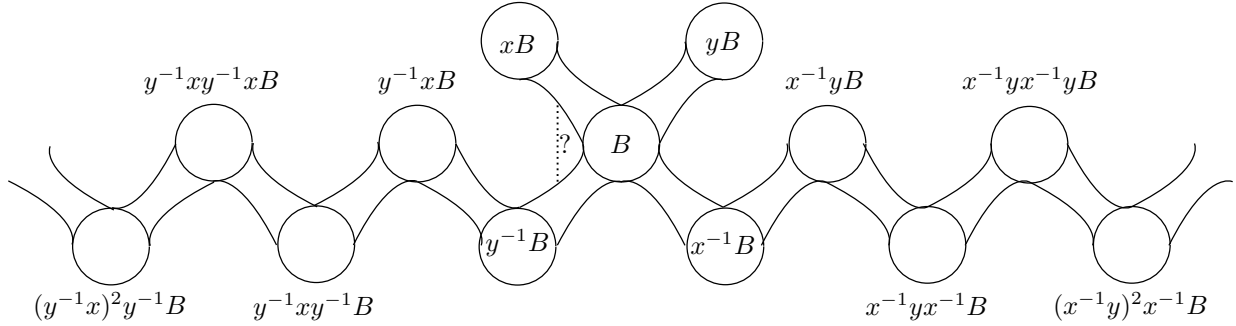

We are now finally ready to prove the theorem (Figure \ref{figure:freetiling} illustrates the proof in the case where $G=\ZZ\ast\ZZ$).  Choose $n$ large enough that $X$ is defined on $n$ and there exists an $n$-axial $\gax\in G$.  Write $B$ for $\BnG$ and $B^{2}$ for $B_{G}(2n,1_{G})$ and let $g$ be some power of $\gax$ such that $g^{k}B^{2}$ is always disjoint from $B^{2}$ for $k\neq 0$---such a $g$ exists because $B^{2}$ only meets finitely many $\gax^{k}B^{2}$. Since $X$ is nonempty, there exists some $\omo\in X$.  Pick distinct integers $m_{1}$ and $m_{2}$ such that $\omo\cdot g^{m_{1}}$ and $\omo\cdot g^{m_{2}}$ agree on $B^{2}$. If we wish to proceed as we did in the case $G=\ZZ$ must find a set of orbit representatives $\fS\subset G$ for the (left) action of $\la g^{m_{2}-m_{1}}\ra$ on $G$ containing $g^{m_{1}}B$ and having properties analogous to the $\fS$ we found for $\ZZ$. We define $\fS$ as follows.

\begin{figure}[t]
\labellist
\small\hair 2pt

\pinlabel $g^{m_{1}}B^{2}$ at 128 349
\pinlabel $g^{m_{2}}B^{2}$ at 513 349
\pinlabel $\fS$ at 256 206
\pinlabel $\omo$ at 449 429
\pinlabel $\sigma$ at 449 183


\endlabellist

\centering
\centerline{\psfig{file=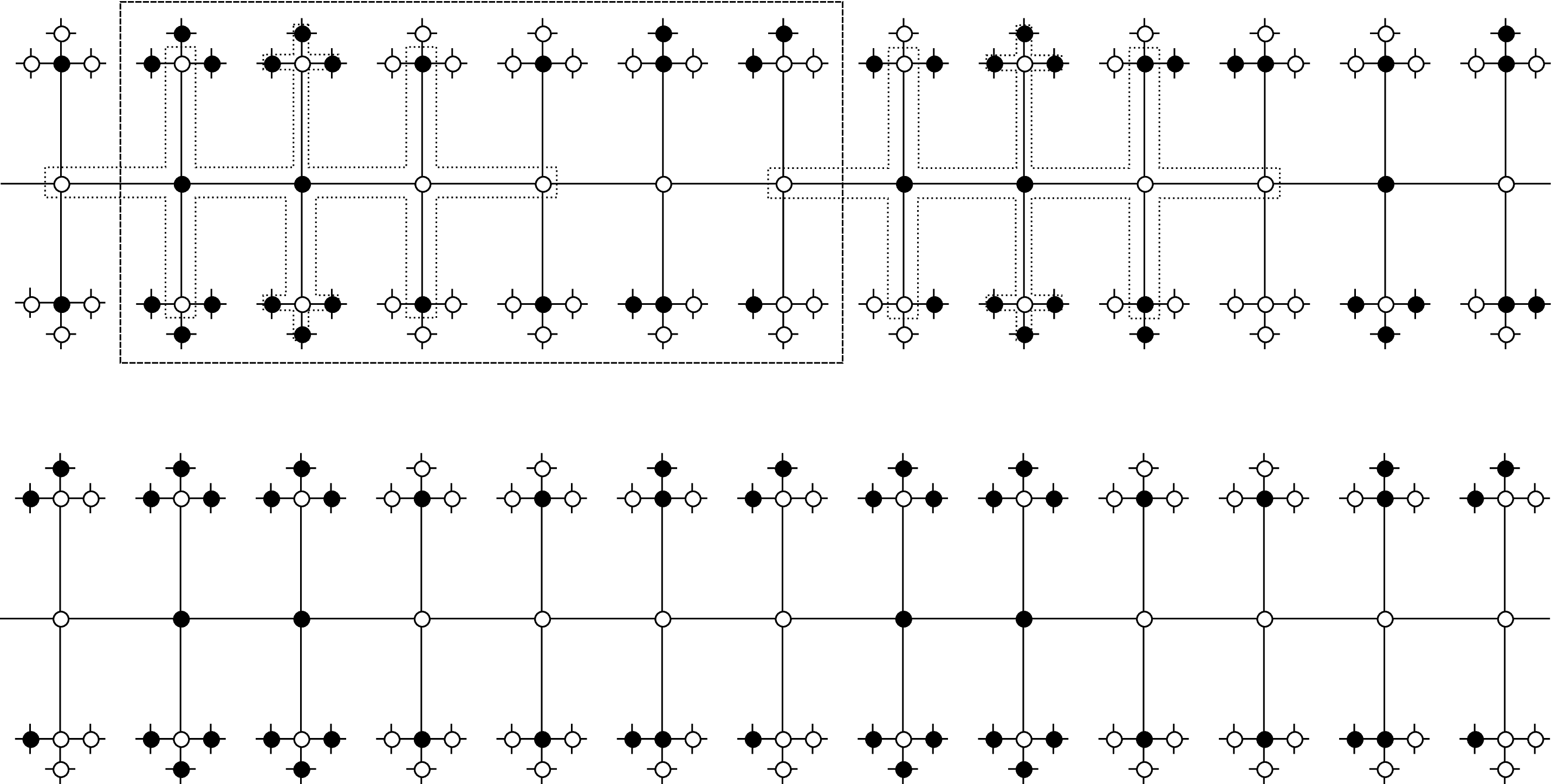,scale=60}}
\caption{If $X$ a subshift of finite type on $\ZZ\ast\ZZ$, and $\omo\in X$, then $\omo$ has the same behavior on two balls $g^{m_{1}}B^{2}$ and $g^{m_{2}}B^{2}$ whose radius is twice the defining radius of $X$. A periodic element $\sigma$ of $X$ is constructed by repeating the pattern realized by $X$ on $\fS$, a fundamental domain for the action of $g^{m_{2}-m_{1}}$ which contains a ball around $g^{m_{1}}$.}
\label{figure:freetiling}
\end{figure}
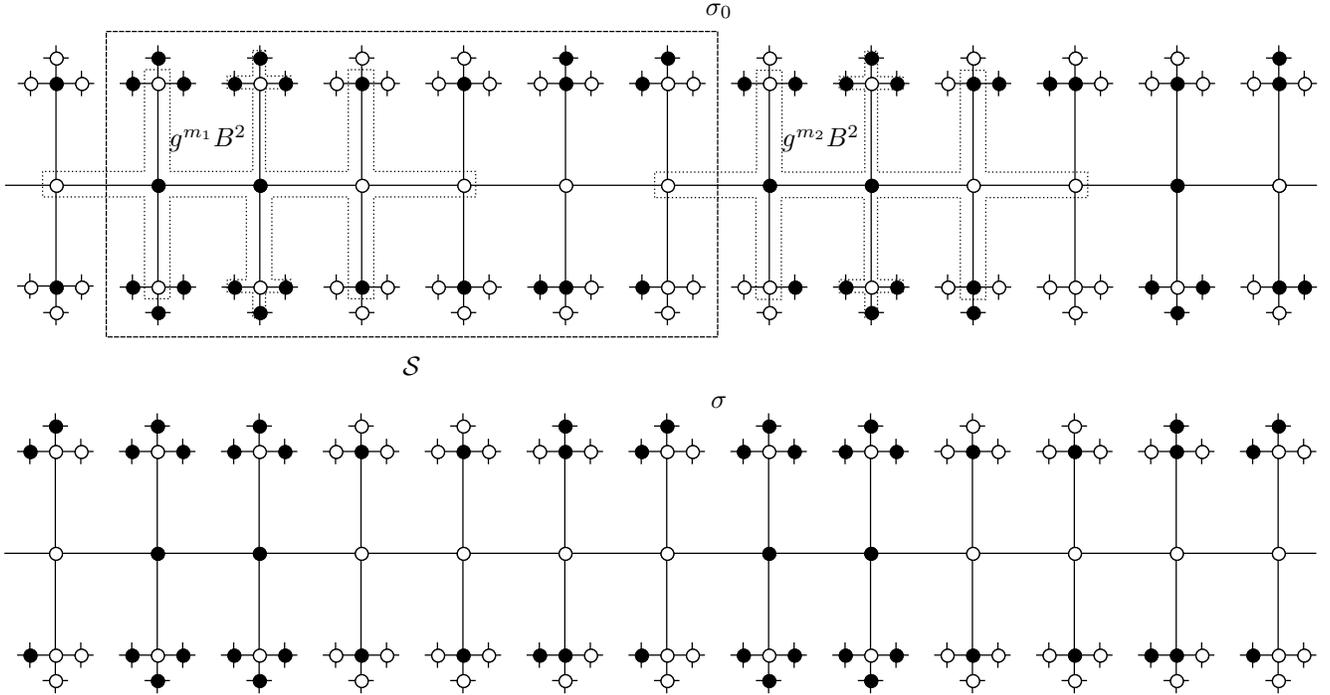

\begin{definition}
\label{definition:funddom}
Let $B_{k}=g^{m_{1}+k(m_{2}-m_{1})} B$ and let $\{\mathcal{C}_{i}\}$ consist of all connected components of $\GmS\smin\bigcup_{k\in\ZZ}B_{k}$. Note that Lemma \ref{lemma:separation} implies that each $\mathcal{C}_{i}$ touches at most two of the $B_{k}$, and these two must have consecutive $k$. We take $\fS$ to be the union of
\begin{itemize}
\item $B_{0}$,
\item the vertex sets of those $\mathcal{C}_{i}$ which touch only $B_{0}$ (and no other $B_{k}$,)
\item and the vertex sets of those $\mathcal{C}_{i}$ which touch both $B_{0}$ and $B_{1}$.
\end{itemize}
\end{definition}

The following lemma enumerates most of the necessary properties of $\fS$.

\begin{lemma}
\label{lemma:funddom}
In the situation of the above paragraph, there exists $\fS\subset G$ such that the following conditions hold.
\begin{itemize}
\item For any integer $k\neq 0$, we have $g^{k(m_{2}-m_{1})}\fS\cap\fS=\emptyset$.
\item For any $h\in G$, there exists an integer $k$ such that $h$ lies in $g^{k(m_{2}-m_{1})}\fS$.
\item For any $h\in G$, there exists an integer $k$ such that
$$hB \subset \fS\cup g^{m_{1}}B^{2}\cup g^{m_{2}}B^{2}.$$
\item $\fS$ contains $g^{m_{1}}B$.
\end{itemize}
\end{lemma}
\begin{proof}
\noindent We now verify that $\fS$ has the desired properties, in order.
\begin{itemize}
\item For a nonzero integer $k$, it is clear that $g^{k(m_{1}-m_{2})}B_{0}$ (which is just $B_{k}$) will not meet $\fS$.  Similarly, if $\fCi$ touches just $B_{0}$, then $g^{k(m_{1}-m_{2})}\fCi$ touches just $B_{k}$, and does not intersect $\fS$. Finally, if some $\fCi$ touches $B_{0}$ and $B_{1}$, then $g^{k(m_{1}-m_{2})}\fCi$ touches $B_{k}$ and $B_{k+1}$, and hence does not intersect $\fS$.
\item Any element of $G$ lies in some $B_{k}$ or some $\fCi$. The translate $g^{-k(m_{1}-m_{2})}B_{k}$ is equal to $B_{1}\subset \fS$. If $\fCi$ meets just $B_{k}$, then $g^{-k(m_{1}-m_{2})}\fCi$ meets just $B_{0}$, and hence lies in $\fS$. If $\fCi$ meets $B_{k}$ and $B_{k+1}$, then $g^{-k(m_{1}-m_{2})}\fCi$ meets $B_{0}$ and $B_{1}$, and is thus a subset of $\fS$.
\item If $x\in B_{k}$, then $g^{-k(m_{2}-m_{1})}B\subset g^{m_{1}}B^{2}$. If $x$ is in some $\fCi$ which touches just $B_{k}$, then any path of length $n$ starting at $x$ must either stay in $\fCi$ or go through $B_{k}$.  Hence, $xB\subset \fCi \cup g^{m_{1}+k(m_{2}-m_{1})}B^{2}$, so $g^{-k(m_{2}-m_{1})}xB\subset g^{m_{1}}B^{2}\cup \fS$. If $x$ lies in some $\fCi$ which touches $B_{k}$ and $B_{k+1}$, then by the same logic, $g^{-k(m_{2}-m_{1})}xB\subset g^{m_{1}}B^{2}\cup \fS\cup g^{m_{2}}\fS$.
\item By definition, $\fS$ contains $g^{m_{1}}B$, which is $B_{0}$.
\end{itemize}
\end{proof}

We now finish the proof of Theorem \ref{theorem:ends}. Take $\fS$ as in the lemma.  For $x\in G$, define $\frm(x)$ to be the $g^{k(m_{2}-m_{1})}$ translate of $x$ which lies in $\fS$. Define $\sigma(x)=\omo(\frm(x))$, so that by definition $\sigma\cdot g^{m_{2}-m_{1}}=\sigma$.  Let $\fSp=\fS\cup g^{m_{1}}B^{2}\cup g^{m_{2}}B^{2}$. As in the $\ZZ$ case, we have that $\sigma$ agrees with $\omo$ on $\fSp$ by the following case by case argument.
\begin{itemize}
\item If $x\in \fS$, then $\sigma(x)=\omo(x)$ by definition.
\item If $x\in g^{m_{1}}B^{2}\smin \fS$, then $x$ lies in some $\fCi$ which touches $B_{0}$ (and possibly also $B_{-1}$,) because there is a path of length at mst $n$ from $x$ to $B_{0}$, and this path cannot pass through any other $B_{k}$ by our assumption that the $g^{k}B^{2}$ are all disjoint.  It follows that either $x$ or $g^{m_{2}-m_{1}}x$ lies in $\fS$, so that we have either
$$\sigma(x)=\omo(x)$$
by definition, or
$$\sigma(x)=\omo(g^{m_{2}-m_{1}}x)=\omo(x),$$
by our assumption that $\omo\cdot g^{m_{1}}$ and $\omo\cdot g^{m_{2}}$ agree on $B^{2}$.
\item If $x\in g^{m_{2}}B^{2}\smin \fS$, then we see similarly that $x$ lies in some $\fCi$ which touches $B_{1}$ (and possibly also $B_{2}$,) and we can proceed in the same fashion.
\end{itemize}

We see that $\sigma\in X$ because for any $x\in G$, Lemma \ref{lemma:funddom} gives us a $k$ such that $g^{k(m_{2}-m_{1})}xB\in\fSp$, and then we have
$$\sigma\cdot x|_{B}=\sigma\cdot g^{-k(m_{2}-m_{1})}x|_{B}=\omo\cdot g^{-k(m_{2}-m_{1})}x|_{B}.$$
Since $X$ is defined on $B$, this establishes the desired result.  We already observed that $\sigma$ is $g^{m_{2}-m_{1}}$ periodic, so we have proved the theorem.
\section{Derivative subshifts.}
\label{section:dshift}
In this section, we will exhibit a subshift of finite type which parameterizes $n$-Lipschitz functions from a finitely presented group $G$ to a finitely generated group $H$, up to translation on $H$ (Theorem \ref{theorem:derivativeshift}). The idea is that an $n$-Lipschitz function $f$ is determined, up to choice of $f(1)$, by its derivative (Definition \ref{definition:derivative} and Figure \ref{figure:derivative}), which is a bounded function from $G\times S$ to $H$ whose value at $(g,s)$ measures the difference between $f(g)$ and $f(gs)$.  The set of such derivatives is shown to be a subshift of finite type when $G$ is finitely presented, by showing that a function on $G\times S$ which looks like a derivative locally can be ``integrated" to give a globally defined $n$-Lipschitz function. Of course, the condition of looking like a derivative locally will be encoded by a finite set of forbidden patterns.  Note that similar subshifts have previously arisen in the literature.  For example Gromov used a subshift parameterizing ``integer 1-cocycles" to code the boundary of a hyperbolic group \cite[\S 3]{cp}.

\paragraph{Notation.} Throughout this section, $G$ will be a group generated by a finite symmetric set $S$, and $H$ will be a group generated by a finite symmetric set $T$. As usual, fixing a finite generating set for a group endows it with a word metric.

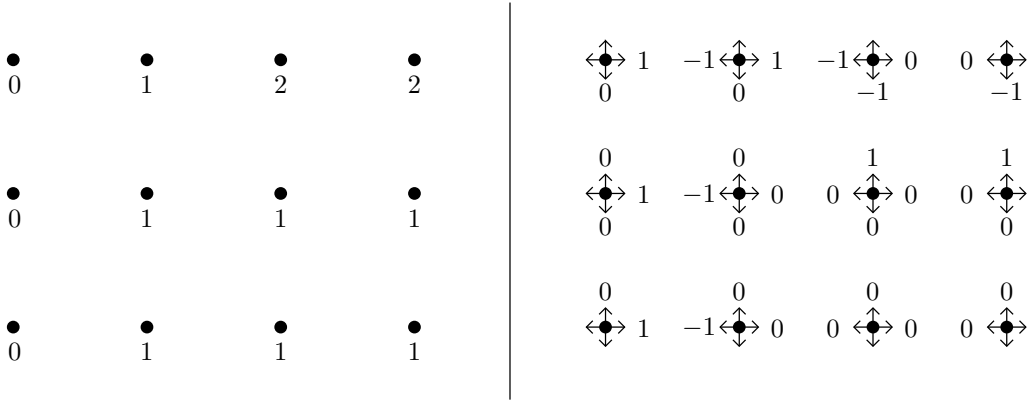
\begin{figure}[t]
\labellist
\small\hair 2pt

\pinlabel $0$ at 5 30  
\pinlabel $0$ at 5 114 
\pinlabel $0$ at 5 198 

\pinlabel $1$ at 88 30   
\pinlabel $1$ at 88 114  
\pinlabel $1$ at 88 198  

\pinlabel $1$ at 172 30  
\pinlabel $1$ at 172 114 
\pinlabel $2$ at 172 198 

\pinlabel $1$ at 256 30  
\pinlabel $1$ at 256 114 
\pinlabel $2$ at 256 198 

\pinlabel $1$ at 400 45  
\pinlabel $1$ at 400 129 
\pinlabel $1$ at 400 213 

\pinlabel $0$ at 484 45  
\pinlabel $0$ at 484 129 
\pinlabel $1$ at 484 213 

\pinlabel $0$ at 568 45  
\pinlabel $0$ at 568 129 
\pinlabel $0$ at 568 213 

\pinlabel $-1$ at 435 45  
\pinlabel $-1$ at 435 129 
\pinlabel $-1$ at 435 213 

\pinlabel $0$ at 519 45  
\pinlabel $0$ at 519 129  
\pinlabel $-1$ at 519 213  

\pinlabel $0$ at 603 45  
\pinlabel $0$ at 603 129 
\pinlabel $0$ at 603 213 

\pinlabel $0$ at 376 109 
\pinlabel $0$ at 376 193 
 
\pinlabel $0$ at 460 109 
\pinlabel $0$ at 460 193 
 
\pinlabel $0$ at 544 109 
\pinlabel $-1$ at 544 193  
 
\pinlabel $0$ at 628 109 
\pinlabel $-1$ at 628 193 

\pinlabel $0$ at 376 68  
\pinlabel $0$ at 376 152 

\pinlabel $0$ at 460 68  
\pinlabel $0$ at 460 152 

\pinlabel $0$ at 544 68  
\pinlabel $1$ at 544 152 

\pinlabel $0$ at 628 68  
\pinlabel $1$ at 628 152 

\endlabellist

\centering
\centerline{\psfig{file=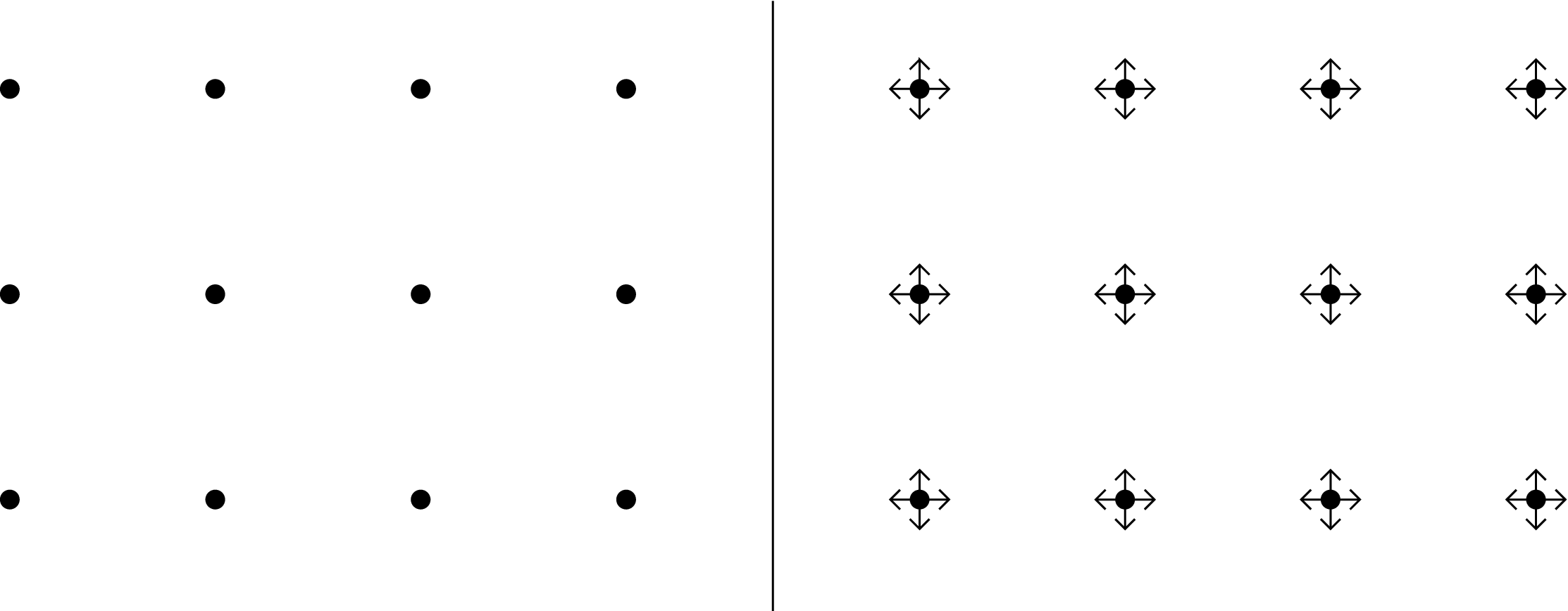,scale=60}}
\caption{At left, a $1$-Lipschitz function $f:\ZZ^{2}\To\ZZ$. At right, its derivative $df$.}
\label{figure:derivative}
\end{figure}

\begin{definition}
\label{definition:derivative}
Fix $n\in \ZZ$ and finitely generated groups $G$ and $H$.  We denote the set of $n$-Lipschitz functions from $G$ to $H$ by $\Lip_{n}(G,H)$.  The derivative is the map $d:\Lip_{n}(G,H)\To (\BHN^{S})^{G}$ which takes $f\in\Lip_{n}(G,H)$ to
$$df:g\mapsto (s\mapsto f(g)^{-1}f(gs)).$$
We write $\la df(g),s\ra$ for $df(g)$ evaluated at $s$.
\end{definition}

See Figure \ref{figure:derivative} for an example.  Observe that $f\in \Lip_{n}(G,H)$ is determined by $f(1)$ and $df$.  We now state the main theorem of this section.

\begin{theorem}
\label{theorem:derivativeshift}
If $G$ is finitely presented, then for any integer $n$ and finitely generated group $H$, we have that
$$\{df:f\in\Lip_{n}(G,H)\}\subset (\BHN^{S})^{G}$$
is a subshift of finite type.
\end{theorem}

\paragraph{A false converse.} In an earlier draft of this article, we conjectured a strong converse to Theorem \ref{theorem:derivativeshift}---namely that if $G$ is such that $\{df:f\in\Lip_{n}(G,\ZZ)\}$ is always a subshift of finite type, then $G$ must be finitely presented. However, Cornulier \cite{cornulier} has pointed out to us some interesting counterexamples. In particular, if $G$ satisfies the finiteness property $FP_{2}$, then all of these subshifts will be of finite type, but it is well known that there are groups which have $FP_{2}$ but are not finitely presented \cite{bestvinabrady}.

\paragraph{Proof of Theorem \ref{theorem:derivativeshift}.} Let $X_{n}$ denote the set $\{df:f\in\Lip_{n}(G,\ZZ)\}$, let $\AAA$ denote $B_{\ZZ}(n,0)^{S}$, and let the natural number $K_{G}\geq 2$ be such that $G$ is presented with respect to $S$ by relators of length at most $K_{G}$. We wish to prove that $X_{n}$ is a subshift of finite type, meaning that membership in $X_{n}$ is determined by some finite list of local conditions.  What sort of local conditions must derivatives satisfy? At least one is immediately obvious, namely we know, for any $g\in G$, that
$$\la df(g),s\ra=\la df(gs),s^{-1}\ra.$$
More generally, if some word $w=s_{0}\ldots s_{k}$ in $S^{\ast}$ is a relation, then we must have that the telescoping product
$$\la df(g),s_{0}\ra \la df(gs_{0}),s_{1}\ra
\ldots\la df(gs_{0}\ldots s_{k-1}),s_{k}\ra$$
represents $1_{H}$ for any $g\in G$. For a fixed $w$, this represents a local condition on $df$, since the product depends only on the values taken by $df$ in $B_{G}(|w|,1)$. Since $G$ is finitely presented, we might hope that $X_{n}$ is defined by a finite set of conditions of this nature, and this is in fact the case. We begin by giving the expected notation for products like the above.



\begin{definition}
\label{definition:integral}
Let $g$ be an element of $G$, let $w\in S^{\ast}$ be some word $s_{0}s_{1}\ldots s_{k}$ (where $s_{i}\in S$,) and let $\sigma$ be an element of $\AAA^{G}$.  We define $\int_{g\cdot w}\sigma$ as the product
$$\la \omg(g),s_{0}\ra\la \omg(gs_{0}),s_{1}\ra\la \omg(gs_{0}s_{1}),s_{2}\ra
\ldots \la \omg(gs_{0}\ldots s_{k-1}),s_{k}\ra.$$
\end{definition}

We now record some properties of this gadget.

\begin{lemma}
\label{lemma:propertiesofintegral}
The integral has the following familiar properties.
\begin{itemize}
\item {\bf Locality.} The value of $\int_{g\cdot w}\sigma$ is determined by $\omg|_{B_{G}(|w|,g)}$.
\item {\bf Additivity.} If $w_{1},w_{2}\in S^{\ast}$, and $h$ is the image of $w_{1}$ in $G$, then
$$\int_{g\cdot w_{1}}\omg\int_{gh\cdot w_{2}}\omg
=\int_{g\cdot w_{1}w_{2}}\omg.$$
\item {\bf Fundamental theorem.} Suppose $f\in\Lip_{n}(G,H)$. Then for $g\in G$ and $w\in S^{\ast}$, we have
$$\int_{g\cdot w}df=f(g)^{-1}f(gw).$$
\end{itemize}
\end{lemma}

\begin{proof}
Locality and additivity follow immediately from Definition \ref{definition:integral}.  The fundamental theorem follows from collapsing the telescoping product
$$\int_{g\cdot w}\omg=\la \omg(g),s_{0}\ra\la \omg(gs_{0}),s_{1}\ra\la \omg(gs_{0}s_{1}),s_{2}\ra
\ldots \la \omg(gs_{0}\ldots s_{k-1}),s_{k}\ra$$
$$=(f(g)^{-1}f(gs_{0}))(f(gs_{0})^{-1}f(gs_{0}s_{1}))(f(gs_{0}s_{1})^{-1}f(gs_{0}s_{1}s_{2}))
\ldots(f(gs_{0}\ldots s_{k-1})^{-1}f(gs_{0}\ldots s_{k}))$$
$$=f(g)^{-1}f(gw).$$
\end{proof}

We will now proceed with the proof of theorem \ref{theorem:derivativeshift}. Let $Y_{n}$ consist of all $\omg\in\AAA^{G}$ such that for any $g\in G$,
$$\int_{g\cdot w_{1}}\omg=\int_{g\cdot w_{2}}\omg$$
whenever the words $w_{1},w_{2}\in S^{\ast}$ are such that $|w_{1}|,|w_{2}|\leq K_{G}$ and $w_{1}$ and $w_{2}$ represent same element of $G$.  The fundamental theorem (Lemma \ref{lemma:propertiesofintegral}) shows that $Y_{n}$ contains $X_{n}$ and locality (Lemma \ref{lemma:propertiesofintegral}) shows that $Y_{n}$ is a subshift of finite type. To prove Theorem \ref{theorem:derivativeshift}, it thus suffices to show that $X_{n}\supset Y_{n}$, i.e., that every element of $Y_{n}$ is the derivative of some element of $\Lip_{n}(G,H)$.



\begin{lemma}
\label{lemma:poincare}
For any $\sigma\in Y_{n}$, the quantity
$$\int_{g\cdot w}\sigma$$
depends only on $g$ and the value $w$ represents in $G$.
\end{lemma}

\begin{proof}
Let $w$ and $w^{\prime}$ be words representing the same element of $G$. Then there exists a homotopy
$$w=w_{0},w_{1},\ldots,w_{k-1},w_{k}=w^{\prime}$$
from $w$ to $w^{\prime}$, meaning a sequence of words $w_{i}\in S^{\ast}$ such that each pair $(w_{i},w_{i+1})$ has the form $(uvx,uv^{\prime}x)$ where $u,v,v^{\prime},x\in S^{\ast}$ are such that $v$ and $v^{\prime}$ have length $\leq K_{G}$ and represent the same element of $G$. But then by repeated application of Lemma \ref{lemma:propertiesofintegral} we have that
$$\int_{g\cdot uvx}\sigma=\int_{g\cdot u}\sigma\int_{gu\cdot v}\sigma\int_{guv\cdot x}\sigma$$
$$=\int_{g\cdot u}\sigma\int_{gu\cdot v^{\prime}}\sigma\int_{guv^{\prime}\cdot x}\sigma=\int_{g\cdot uv^{\prime}x}\omg.$$
It follows that $\int_{g\cdot w}\omg=\int_{g\cdot w^{\prime}}\omg$ as desired.
\end{proof}

Given $\omg\in Y_{n}$, we may now define a function $f\in\Lip_{n}(G,H)$ with derivative $\omg$ by taking $f(g)$ to be
$$\int_{1_{g}\cdot w}\omg$$
for any $w$ representing $g$ (by Lemma \ref{lemma:poincare}, the choice of $w$ is irrelevant). We can see that $f$ is $n$-Lipschitz by the fact that for words $w_{1},w_{2}$ representing $g_{1},g_{2}\in G$ respectively, we have
$$d(f(g_{1}),f(g_{2}))=f(g_{1})^{-1}f(g_{2})=\left|\left(\int_{1_{G}\cdot w_{1}}\omg\right)^{-1}\int_{1_{G}\cdot w_{2}}\omg\right|_{T}$$
$$=\left|\int_{g_{1}\cdot w_{1}^{-1}w_{2}}\omg\right|_{T}=\left|\int_{g_{1}\cdot w}\omg\right|_{T}\leq nd(g_{1},g_{2})$$
for a geodesic word $w$ representing $g_{1}^{-1}g_{2}$ (we have used the fact that $\int_{g_{1}\cdot w_{1}^{-1}}\omg$ is the inverse of $\int_{1_{G}\cdot w_{1}}\omg$ for $\omg\in Y_{n}$, which follows from Lemma \ref{lemma:poincare}.) We see that $df=\omg$ by similar reasoning. Hence, we have shown that the subshift of finite type $Y_{n}$ is equal to $X_{n}$, thus establishing Theorem \ref{theorem:derivativeshift}.

\section{Parameterizing quasi isometries.}
\label{section:qi}
In this section, we shall define a subshift of finite type which codes certain quasi isometries between two finitely presented groups $G$ and $H$ (Lemma \ref{lemma:qipairshift}). This technical result will be used in all of our main theorems. In Subsection \ref{subsection:weak}, we shall see that if some point of this subshift is stabilized by a subgroup $G_0\subset G$, then this point corresponds to a quasi isometry $G\To H$ whose restriction to $G_0$ is a homomorphism (Proposition \ref{proposition:periods}). This will imply one of our main theorems (Theorem \ref{theorem:mainthmweak}). Namely, if $G$ is finitely presented and torsion free, and $G$ has no weakly aperiodic subshift of finite type, then we shall see that $\{G\}$ is QI-rigid (Corollary \ref{corollary:weakaperiodicity}).

Throughout this section $G$ and $H$ will be finitely presented groups equipped with fixed finite generating sets $S$ and $T$, which induce word metrics on $G$ and $H$. We would like to say that the set of $df$ such that $f:G\To H$ is an $n$-quasi isometry forms a subshift of finite type, but it is far from clear that one can verify any sort of quasi injectivity or properness from the local behavior of $df$. Instead, we shall look at the set of pairs $(df,\lFf)$ where $f$ is an $n$-Lipschitz quasi isometry and $\lFf$ records the local behavior of an $n$-Lipschitz two-sided $n$-quasi inverse $F:H\To G$. We formalize this as follows.

\begin{definition}
\label{definition:qipair}
An $n$-QI pair is a pair of functions $(f,F)$ where
\begin{itemize}
\item $f:G\To H$ and $F:H\To G$ are $n$-Lipschitz.
\item $F$ is a left $n$-quasi inverse to $f$, i.e., $d(F\circ f (g),g)\leq n$ for all $g\in G$.
\item $F$ is a right $n$-quasi inverse to $f$, i.e., $d(f\circ F (h),h)\leq n$ for all $h\in H$.
\end{itemize}
The set of $n$-QI pairs is denoted $\QIP_n(G,H)$, suppressing the dependence on choice of generating sets $S$ and $T$.
\end{definition}

The following proposition shows that $f:G\To H$ is a member of some QI-pair if and only if it is a quasi isometry.

\begin{proposition}
\label{proposition:qipair}
If $f:G\To H$ is a quasi isometry, then there exist $n\in\NN$ and $F:H\To G$ such that $(f,F)\in \QIP_n(G,H)$.

Conversely, if $(f,F)\in \QIP_n(G,H)$, then $f$ is a quasi isometry.
\end{proposition}

\begin{proof}
Suppose $(f,F)\in\QIP_{n}(G,H)$.  Then for any $x,y\in G$,
$$d(x,y)\leq d(x,(F\circ f)(x))+d((F\circ f)(x),(F\circ f)(y))+d((F\circ f)(y),y)\leq 2n+n d(f(x),f(y)),$$
and thus:
$$d(f(x),f(y))\geq \frac{d(x,y)}{n}-2.$$
Since $f$ is Lipschitz, this implies that $f$ is a quasi isometric embedding.  But $f$ is quasi surjective because for $h\in H$ we have $d(f(F(h)),h)\leq n$.  Hence, $f$ is a quasi isometry.

Conversely, suppose $f$ is an $N$-quasi isometry.  We will now define $F:H\rightarrow G$ such that $(f,F)$ is a QI pair. Using $N$-quasi surjectivity, for each $h\in H$, choose an $F(h)$ in $G$ such that $d(f(F(h)),h)\leq N$.  Since $f$ is an $N$-quasi isometric embedding, we know that for all $h_{1},h_{2}\in H$, we have
$$\frac{d(F(h_{1}),F(h_{2}))}{N}-N\leq d(f(F(h_{1})),f(F(h_{2})))
\leq d(f(F(h_{1})),h_{1})+d(h_{1},h_{2})+d(h_{2},f(F(h_{2})))$$
$$\leq 2N+d(h_{1},h_{2}).$$
Hence
$$d(F(h_{1}),F(h_{2}))\leq Nd(h_{1},h_{2})+3N^{2},$$
so $F$ is $3N^{2}+N$-Lipschitz (as $1$ is the smallest positive distance in $H$). We can see that $f$ is $2N$-Lipschitz because, since $1$ is the smallest positive distance in $G$,
$$d(f(x),f(y))\leq Nd(x,y)+N\leq 2Nd(x,y).$$

By definition, $F$ is a right $N$-quasi inverse to $f$.  To see that it is a left quasi inverse, note that for $g\in G$ we have
$$d(f((F\circ f)(g)),f(g))=d((f\circ F)(f(g)),f(g))\leq N.$$
It follows that
$$d((F\circ f)(g),g)\leq Nd(f((F\circ f)(g)),f(g))+N\leq N^{2}+N.$$
So, taking $n$ to be greater than each of $\{2N, 3N^{2}+N, N, N^{2}+N\}$, we have that $(f,F)$ is an $n$-QI pair.
\end{proof}

Given and $n$-QI pair $(f,F)$, we now define a function $\lFf$ on $G$ whose value at $g$ encodes the local behavior of $F$ near $f(g)$. Our technical lemma will state that the collection of all $(df,\lFf)$ forms a subshift of finite type.

\begin{definition}
\label{definition:lFf}
If $(f,F)\in \QIP_n(G,H)$, then define
$$\lFf:G\To B(n^{2}+n,1_{G})^{B(n,1_{H})}$$
by setting, for $g\in G$ and $k\in B(n,1_{H})$,
$$(\lFf(g))(k)=g^{-1}F(f(g)k).$$
We shall write $\la \lFf(g), k \ra$ for $(\lFf(g))(k).$
\end{definition}

In other words, $\lFf$ records the values of $F$ on $B(n,f(g))$ (relative to $g$). The following proposition implies that we can always replace $(f,F)\in \QIP_n(G,H)$ by some other $n$-QI pair $(\tilde{f},\tilde{F})$ such that $\tilde{f}(1_{G})=1_{H}$ and $(d\tilde{f},\ell_{\tilde{F}\tilde{f}})$ remains equal to $(df,\lFf)$.

\begin{proposition}
\label{proposition:translation}
Suppose $(f_0,F_0)\in \QIP_n(G,H)$. Define, for any $h_0\in H$, functions $f_1:G\To H$ and $F_1:H\To G$ by
$$f_1(g)=h_0 f_0(g)\quad,\quad F_1(h)=F_0(h_0^{-1}h).$$
Then $(f_1,F_1)\in\QIP_n(G,H)$, with $df_1=df_0$ and $\ell_{F_0f_0}=\ell_{F_1f_1}$.
\end{proposition}

The proof is left to the reader. We can now state the key lemma of this section.

\begin{lemma}
\label{lemma:qipairshift}
Let $G$ and $H$ be finitely presented groups, and let
$$\AAA=B(n,1_{H})^{S}\times B(n^{2}+n,1_{G})^{B(n,1_{H})}.$$
The set
$$\{(df,\lFf):(f,F)\in \QIP_{n}(G,H)\}\subset \AAA^{G}$$
is a subshift of finite type.  A finite set of forbidden patterns defining this subshift may be computed when $G$ has decidable word problem---that is for fixed $G$ and $H$, with $G$ having decidable word problem, there is an algorithm which consumes $n$ and determines a finite set of forbidden patterns which define the desired subshift.
\end{lemma}

\begin{proof}
Fix $M\in\NN$ strictly greater than $n$ and $K_{H}$ (the length of the longest relator of $H$). Fix $N\in \NN$ be strictly greater than $nM$. Obviously we may assume that $G$ is quasi isometric to $H$, as the empty set is certainly a subshift of finite type. If $G$ has decidable word problem, then so does $H$, as the word problem is a QI invariant for finitely presented groups \cite[Theorem 2.2.5]{wp}.

We will presently define a set $X\subset \AAA^{G}$. We shall then show the following, in no particular order.
\begin{itemize}
\item $X$ is a subshift of finite type.
\item If $G$ has decidable word problem, we can compute forbidden patterns for $X$.
\item $X$ contains $\{(df,\lFf):(f,F)\in \QIP_{n}(G,H)\}$.
\item Every element $\sigma=(\omd,\oml)$ of $X$ is of the form $(df,\lFf)$ for some $n$-QI pair $(f,F)$.
\end{itemize}

\begin{definition}
Let $X$ be the set of all $\sigma=(\omd,\oml)\in\AAA^{G}$ satisfying the following two conditions.

First, $\omd=df$ for some $f\in \Lip_n(G,H)$. 

Second, letting $f\in \Lip_n(G,H)$ be such that $f(1_{G})=1_{H}$ and $df=\omd$, we have that for all $g\in G$, there is a function
$$\Fomg:\Nn_M f(B(N,g))\To B(n+N+nM,g)$$
such that the following properties hold.
\begin{itemize}
\item $\Fomg$ is $n$-Lipschitz
\item $\Fomg$ is a left quasi inverse to $f$, i.e., $\Fomg\circ f(g')\in B(n,g')$ for $g'\in B(N,g)$.
\item $\Fomg$ is a right quasi inverse to $f$, i.e., $f\circ\Fomg(h)\in B(n,h)$ for $h\in \Nn_M f(B(N,g))$.
\item $\Fomg$ is compatible with $\oml$ on $B(n+N+nM+n^{2},g)$, i.e, for $g'\in B(n+N+nM+n^{2},g)$, if $f(g')k\in \Nn_M f(B(N,g))$ for some $k\in B(n,1_{H})$, then
$$g'\la\oml(g),k\ra=\Fom(f(g')k).$$
\end{itemize}
\end{definition}

We now verify that $X$ has the desired properties.

\begin{proposition}
$X$ contains all $(df,\lFf)$ such that $(f,F)$ is an $n$-QI pair.
\end{proposition}
\begin{proof}
By definition of an $n$-QI pair, $(df,\lFf)$ satisfies the first defining condition of $X$. Taking $\Fomg=F$, we see that it also satisfies the second defining condition.
\end{proof}

\begin{proposition}
$X$ is a subshift of finite type. If $G$ has decidable word problem, we may compute forbidden patterns for $X$.
\end{proposition}
\begin{proof}
By (the proof of) Theorem \ref{theorem:derivativeshift}, to ensure that $\omd=df$ for some $f\in \Lip_n(G,H)$, it suffices to check that
$$\int_{g}^{gw_{1}}\omd=\int_{g}^{gw_{2}}$$
whenever $w_{1},w_{2}\in S^*$ represent the same element of $G$ and have length at most $K_{G}$ (the length of the longest relator of $G$).

To see that the second defining condition of $X$ is local, it suffices to rephrase the desired properties for $\Fomg$ in terms of the function
$$L_{\sigma g}:h\mapsto g^{-1}\Fomg(f(g)h).$$
For instance, the domain of $L_{\sigma g}$ can be written as
$$\Nn_M \left\{\int_{g}^{g'}df:g'\in B(N,g)\right\},$$
which can be determined from $\sigma|_{B(N,g)}$, and the condition that $\Fomg$ is a left quasi inverse to $f$ can be rephrased as
$$L_{\sigma g}\left(\int_{g}^{g'}df\right)\in B(n,g^{-1}g')$$
for $g'\in B(N,g)$, which can be checked from $\sigma|_{B(N,g)}$. In this way, forbidden patterns enforcing the second defining condition may be defined on the ball $B$ in $G$ of radius $n+N+nM+n^{2}$.

When $G$ has decidable word problem, we can actually construct $B$ via Turing machine, and enumerate all the possibilities for $\sigma|_{B}$. For each of these possibilities, we may algorithmically enumerate possible $L_{\sigma g}$ and then check whether they satisfy the given conditions. Hence, we may compute forbidden patterns for $X$.
\end{proof}

\begin{proposition}
If $\sigma=(\omd,\oml)\in X$, then there exists $(f,F)\in\QIP_n(G,H)$ such that $\sigma=(df,\lFf)$.
\end{proposition}
\begin{proof}
As usual, let $f\in \Lip_n(G,H)$ be such that $f(1_{G})=1_{H}$ and $df=\omd$. In order to describe our proof strategy, we need a few definitions.

\begin{definition}
Let $\Delta=\{(g\la \oml(g),k\ra,f(g)k)\in G\times H\}$. Let $\RRD$ be the graph with vertices $\Delta$ and an edge connecting $(g,h)$ to $(g',h')$ whenever
\begin{itemize}
\item $d(h,h')=1,$
\item and $d(g,g')\leq N$.
\end{itemize}
\end{definition}

We need to show that $\Delta$ is the graph of some function $F:H\To G$ and that $(f,F)\in\QIP_n(G,H)$. To show that $\Delta$ is a graph, it suffices to show that $\RRD\To\Cay_T H$ is a covering space and that every relator of $H$ lifts to a loop in $\RRD$.

\paragraph{$\RRD\To \Cay_T H$ is a covering space.} It suffices to show that the neighbors of each vertex of $R\Delta$ exactly correspond to the elements of $T$.  Let $(g\la\oml(g),k \ra,f(g)k)$ be a vertex of $R\Delta$, we will see that its neighbors are exactly
$$\{(\Fomg(f(g)kt),f(g)kt):t\in T\}.$$
Certainly, these are all neighbors of $(g\la\oml(g),k\ra,f(g)k)$, and if $(g'\la \oml(g'),k'\ra,f(g')k')$ is some other neighbor, then $f(g')k'=f(g)kt$ for some $t\in T$ and $g'\in B(N,g)$. But then by the defining conditions of $X$, we must have that $g'\la \oml(g'),k'\ra=\Fomg(f(g)kt)$ as desired.

\paragraph{Relators of $H$ lift to loops in $\RRD$.} Let $w=s_{0}s_{1}\ldots s_{m}$ and $w'=s'_{0}s'_{1}\ldots s'_{m'}$ be words in $T$ of length at most $K_{H}$, where $K_{H}$ is the length of the longest relator of $H$, such that $w$ and $w'$ represent the same element $k$ of $H$. Let $(g,h)$ be some vertex of $\Delta$. Then by the defining conditions of $X$, and the fact that $M>K_{H}$, we have that the paths with vertices
$$(g,h),(\Fomg(hs_{0}),hs_{0}),(\Fomg(hs_{0}s_{1}),hs_0s_1),\ldots
(\Fomg(hs_0\ldots s_m),hs_{0}\ldots s_m)$$
and
$$(g,h),(\Fomg(hs'_{0}),hs'_{0}),(\Fomg(hs'_{0}s'_{1}),hs'_0s'_1),\ldots
(\Fomg(hs'_0\ldots s'_m),hs'_{0}\ldots s'_{m'})$$
share the same endpoints and project to the paths starting at $h$ labeled by $w$ and $w'$ respectively. In other words, the loop at $h$ formed by $w$ followed by the reverse of $w'$ lifts to $\RRD$.

It follows that every loop in $\Cay_T H$ lifts to $\RRD$. Hence, the fibers of $\RRD\To \Cay_T H$ have cardinality $1$, i.e., $\Delta$ is a graph of some function $F:H\To G$. Since $F$ must locally agree with the $\Fomg$, we see by the defining conditions of $X$ that $(f,F)$ is an $n$-QI pair.\end{proof}

We have completed the proof of the lemma.\end{proof}

\subsection{Weak aperiodicity.}
\label{subsection:weak}
We shall now apply Lemma \ref{lemma:qipairshift} to show if $G$ is finitely presented with no weakly aperiodic subshift of finite type, then $\{G\}$ must be QI-rigid (up to taking quotients by finite groups). The idea is that if there is some $n$-QI pair $(f,F)$ such that $f(1_{G})=1_{H}$ and $(df,\lFf)$ has finite index stabilizer in $G$, then $f$ restricts to an isomorphism from some finite index subgroup of $G$ to a finite index subgroup of $H$. To see this, we first need the following proposition.

\begin{proposition}
\label{proposition:periods}
If $(df,\lFf)$ is $\pi$ periodic for some $(f,F)\in \QIP_n(G,H)$ and $\pi\in G$, then for all $g\in G$,
$$f(\pi g)=f(\pi)f(1_{G})^{-1}f(g).$$
\end{proposition}

\begin{proof}
$$f(\pi g)=f(\pi)\int_{\pi}^{\pi g} df=f(\pi)\int_{1}^{g}(df\cdot \pi)$$
$$=f(\pi)\int_{1}^{g}df=f(\pi)f(1_{G})^{-1}f(g).$$
\end{proof}

\begin{corollary}
\label{corollary:weakaperiodicity}
If $G$ is a finitely presented group with no weakly aperiodic subshift of finite type, and $H$ is quasi isometric to $G$, then there exist finite index subgroups $G_0\subset G$ and $H_0\subset H$ such that $H_0$ is isomorphic to the quotient of $G_0$ by a finite subgroup $K$.
\end{corollary}

\begin{proof}
For $n$ sufficiently large, we have $\QIP_n(G,H)$ nonempty. By Lemma \ref{lemma:qipairshift}, the set
$$\{(df,\lFf):(f,F)\in \QIP_{n}(G,H)\}$$
is then a nonempty subshift of finite type. By assumption, some point $(df,\lFf)$ of this subshift must be fixed by some finite index subgroup $G_{0}\subset G$. By Proposition \ref{proposition:translation}, we may assume without loss of generality that $f(1_{G})=1_{H}$. By Proposition \ref{proposition:periods}, the restriction of $f$ to $G_{0}$ is a homomorphism. Because $f$ is a quasi isometry, the kernel $K$ of $f|_{G_{0}}$ is finite and the image $H_0=f(G_{0})$ is finite index in $H$.
\end{proof}

\section{Pullbacks and QI-invariance.}
\label{section:pullbacks}
In this final section, we prove our remaining main theorems.
\begin{itemize}
\item For finitely presented groups, having decidable domino problem is a QI-invariant (Corollary \ref{corollary:dominoproblem})
\item For finitely presented torsion free groups, having a strongly aperiodic subshift of finite type is a QI-invariant (Corollary \ref{corollary:strongaperiodicity}).
\end{itemize}
The key construction is given by Lemma \ref{lemma:pullbackshift}. This will be a subshift of finite type parameterizing tuples $(f,F,\sigma)$ where $(f,F)\in \QIP_n(G,H)$ and $\sigma$ is a configuration of some subshift of finite type $X_{H}$ on $H$. This ``pullback subshift" will be empty exactly when $X_{H}$ is, allowing us to solve the domino problem on $H$ by solving it on $G$. Furthermore, periodic states of the pullback subshift will correspond to periodic states of $X_{H}$, so it will be strongly aperiodic when $X_{H}$ is. In order to describe the construction, we need a few definitions.

\begin{definition}
\label{definition:higherblock}
If $\sigma\in A^{H}$, then $B_n\sigma \in (A^{B(n,1_{H})})^{H}$ is given by setting, for $(h,k)\in H\times B(n,1_{H})$,
$$(B_n\sigma(h))(k)=\sigma(hk).$$
We will write $\la B_n\sigma(h), k\ra$ for $(B_n\sigma(h))(k)$
\end{definition}

This is an instance of the higher block subshifts considered in \cite{carpen}, and in particular, if $X_{H}\subset A^{H}$ is a subshift of finite type, then so is $B_n(X_{H})$, although we shall not need this fact.

\begin{definition}
\label{definition:pullback}
Given $\sigma\in A^{H}$ and $f:G\To H$, let
$$f^*\sigma = \sigma\circ f.$$
We say that $f^* \sigma$ is the pullback of $\sigma$ under $f$.
\end{definition}

Of course, if $X_{H}$ is a subshift of finite type, and $f:G\To H$ a quasi isometry, it is unlikely that $f^* X_{H}$ is itself a subshift of finite type. In fact, $f^* \sigma$ may not even see all the data of $\sigma$---as $f$ does not have to be surjective---although $f^*B_n\sigma$ certainly will. To produce a subshift of finite type on $G$ from $X_{H}$, it is necessary to consider the pullbacks of $\sigma\in X_{H}$ under all possible $f$ such that $(f,F)\in \QIP_n(G,H)$.

\begin{lemma}
\label{lemma:pullbackshift}
Let $G$ and $H$ be finitely presented groups, and $X_{H}\subset A^{H}$ a subshift of finite type. Let
$$\AAA=B(n,1_{H})^{S}\times B(n^{2}+n,1_{G})^{B(n,1_{H})}\times A^{B(n,1_{H})}.$$
Then
$$\{(df,\lFf,f^*(B_{n}\sigma)):(f,F)\in \QIP_{n}(G,H),\sigma\in X_{H}\}\subset
\AAA^{G}$$
is a subshift of finite type. A finite set of forbidden patterns defining this subshift may be computed from forbidden patterns defining $X_{H}$ when $G$ has decidable word problem.
\end{lemma}

\begin{proof}
Note that, as $G$ is finitely presented, it will have decidable word problem if and only if $H$ does.

Suppose $\sigma\in\AAA^{G}$. By Lemma \ref{lemma:qipairshift}, there exists a finite set of forbidden patterns (computable when $G$ has decidable word problem) which ensures that $\sigma$ has the form $(df,\lFf,\sigma_{X})$ for some $n$-QI pair $(f,F)$. Thus, we must find local rules which ensure that $\sigma_X$ is of the form $f^*(B_n \sigma_{0})$ for some $\sigma_0 \in X_{H}$.

First, we mandate that whenever $f(g_1)k_1=f(g_2)k_2$ for some $g_1,g_2\in G$ and $k_1,k_2\in B(n,1_{H})$, then $\la\sigma_X(g_1),k_1\ra=
\la\sigma_X(g_2),k_2\ra$. This is a local rule because the equality $f(g_{1})k_{1}=f(g_{2})k_{2}$ implies that
$$d(F\circ f(g_{1}),F\circ f(g_{2}))\leq 2n^{2},$$
and hence $d(g_{1},g_{2})\leq 2n^{2}+2n$. If $G$ and $H$ have decidable word problem, we can compute forbidden patterns enforcing this rule because we can algorithmically check whether $\int_{g_{1}}^{g_{2}}df=k_{1}^{-1}k_{2}$ from the local values of $df$. Any $\sigma_X$ following this rule can be written as $f^*(B_n \sigma_{0})$ for some unique $\sigma_0\in A^H$. In particular, take $\sigma_{0}(h)$ to be $\la\sigma_X(F(h)),(f\circ F(h))^{-1}h\ra.$

To force $\sigma_0$ to be in $X_{H}$, we simply mandate that if
$$\alpha:\{h_{1},\ldots,h_{m}\}\To A$$
is a defining forbidden pattern for $X_{H}$, then for any set $\{(g_{1},k_{1}),\ldots,(g_{m},k_{m})\}\subset G\times B(n,1_{H})$ such that
$$(f(g_{i})k_{i})^{-1}(f(g_{j})k_{j})=h_{i}^{-1}h_j$$
for all $i$ and $j$, there must be some $i$ such that $\la \sigma_X(g_{i}),k_{i}\ra\neq \alpha(h_{i})$---in other words, if some points $f(g_i)k_i$ form a translate of the defining set of the forbidden pattern $\alpha$, then when we try to reconstruct the values taken by $\sigma_0$ at those points, we do not see $\alpha$. This is a local rule by the same considerations as before, and certainly computable when $G$ and $H$ have decidable word problem. By imposing such a rule for each defining pattern of $X_{H}$, we ensure that $\sigma_X$ is in $f^*(B_n X_H)$ as desired.
\end{proof}

\begin{corollary}
\label{corollary:strongaperiodicity}
If $G$ and $H$ are finitely presented, torsion free groups, and $H$ has a strongly aperiodic subshift of finite type, then $G$ also has a strongly aperiodic subshift of finite type.
\end{corollary}

\begin{proof}
Let $X_{H}\subset A^{H}$ be a strongly aperiodic subshift of finite type, and let $X_{G}\subset \AAA^{G}$ be
$$\{(df,\lFf,f^*(B_{n}\sigma)):(f,F)\in \QIP_{n}(G,H),\sigma\in X_{H}\},$$
where $\AAA$ is as in the statement of Lemma \ref{lemma:pullbackshift}. We will show that $X_{G}$ is strongly aperiodic.

Suppose some point $(df,\lFf,f^*(B_{n}\sigma))$ has a period $\pi\in G$. By Proposition \ref{proposition:periods}, we have that $f(\pi g)=f(\pi)f(1_{G})^{-1}f(g)$ for all $g\in G$. Letting $h_{\pi}$ be $f(\pi)f(1_{G})^{-1}$, periodicity of $f^*(B_{n}\sigma)$ now tells us that for all $g\in G$,
$$(B_{n}\sigma)(f(g))=(B_{n}\sigma)(f(\pi g))=(B_{n}\sigma)(h_{\pi}f(g)),$$
and consequently, for any $k\in B(n,1_{H})$,
$$\sigma(f(g)k)=\sigma(h_{\pi}f(g)k).$$
Since every element of $H$ can be written as such a product $f(g)k$, it follows that $\sigma$ is $h_{\pi}$-periodic, and we must have $h_{\pi}=1_{H}$. As $(df,\lFf,f^*(B_{n}\sigma))$ will still be periodic under any power of $\pi$, we must have $h_{\pi^{k}}=1_{H}$, for any $k$. The set $f^{-1}\{1_{H}\}$ is finite because $f$ is a quasi isometry, so $\{\pi^{k}\}$ is finite, and $\pi$ must equal $1_{G}$ by our assumption of torsion freeness.
\end{proof}

\begin{corollary}
\label{corollary:dominoproblem}
If $G$ and $H$ are finitely presented groups, with $G$ quasi isometric to $H$, and $G$ has decidable domino problem, then $H$ has decidable domino problem.
\end{corollary}

\begin{proof}
If $G$ has undecidable word problem, then $H$ also has undecidable word problem \cite[Theorem 2.2.5]{wp}, and therefore $G$ and $H$ both have undecidable domino problem as it is well known (see for instance the Introduction of \cite{aubrunkari}) that a group with undecidable word problem also has undecidable domino problem. Consequently, we may assume without loss of generality that $G$ and $H$ have decidable word problem.

Now, suppose we have an algorithm solving the domino problem on $G$. For any $n$, we can check whether $\QIP_n(G,H)$ is nonempty, since by Lemma \ref{lemma:qipairshift} we can compute forbidden patterns defining the subshift of finite type
$$\{(df,\lFf)|(f,F)\in \QIP_n(G,H)\}$$
and the domino problem on $G$ lets us determine if this is empty. As $G$ and $H$ are quasi isometric, we can compute an $n$ such that $\QIP_n(G,H)$ is nonempty (by trying larger and larger $n$). Given forbidden patterns defining a subshift of finite type $X_{H}\subset A^{H}$, by Lemma \ref{lemma:pullbackshift} we may compute forbidden patterns for the subshift of finite type $X_{G}\subset \AAA^{G}$ given by
$$\{(df,\lFf,f^*B_n(\sigma))|(f,F)\in \QIP_n(G,H)\},$$
which will be empty if and only if $X_{H}$ is. Since we can solve the domino problem on $G$, we can algorithmically determine emptiness of $X_{G}$, and hence of $X_{H}$.
\end{proof}

\bibliographystyle{plain}
\bibliography{bibliography}

\noindent
David Bruce Cohen\\
Department of Mathematics\\
University of Chicago\\
5734 S. University Avenue,\\
Room 208C\\
Chicago, Illinois 60637\\
E-mail: {\tt davidbrucecohen@gmail.com}\\

\end{document}